\documentclass{article}%
\usepackage{amsfonts}
\usepackage{amssymb}
\usepackage{graphicx}
\usepackage{amsmath}
\usepackage{amscd}%
\newtheorem{theorem}{Theorem}

\newtheorem{corollary}[theorem]{Corollary}

\newtheorem{definition}[theorem]{Definition}

\newtheorem{lemma}[theorem]{Lemma}

\newtheorem{proposition}[theorem]{Proposition}

\newenvironment{proof}[1][Proof]{\textbf{#1.} }{\ \rule{0.5em}{0.5em}}
\begin{document}

\title{Binary matroids and local complementation}
\author{Lorenzo Traldi\\Lafayette College\\Easton, Pennsylvania 18042 USA}
\date{}
\maketitle

\begin{abstract}
We introduce a binary matroid $M[IAS(G)]$ associated with a looped simple
graph $G$. $M[IAS(G)]$ classifies $G$ up to local equivalence, and determines
the delta-matroid and isotropic system associated with $G$. Moreover, a
parametrized form of its Tutte polynomial yields the interlace polynomials of
$G$.

\bigskip


\bigskip

Mathematics Subject\ Classification. 05B35, 05C31

\end{abstract}

\section{Introduction}

A \emph{graph }$G=(V(G),E(G))$ consists of a finite vertex-set $V(G)$ and a
finite edge-set $E(G)$. Each edge is incident on one or two vertices; an edge
incident on only one vertex is a \emph{loop}. The two vertices incident on a
non-loop edge are \emph{neighbors}, and the \emph{open neighborhood} of a
vertex $v$ is $N(v)=\{$neighbors of $v\}$. A graph in which different edges
can be distinguished by their vertex-incidences is a \emph{looped simple
graph}, and a \emph{simple graph} is a looped simple graph with no loop.

In this paper we are concerned with properties of looped simple graphs
motivated by two sets of ideas. The first set of ideas is the theory of the
principal pivot transform (PPT) over $GF(2)$. PPT over arbitrary fields was
introduced more than 50 years ago by Tucker \cite{Tu}; see also the survey of
Tsatsomeros \cite{Ts}. According to Geelen \cite{G}, PPT transformations
applied to the mod-2 adjacency matrices of looped simple graphs are generated
by two kinds of \emph{elementary} PPT operations, \emph{non-simple local
complementations} with respect to looped vertices and \emph{edge pivots} with
respect to edges connecting unlooped vertices. The second set of ideas is the
theory of 4-regular graphs and their Euler circuits, initiated more than 40
years ago by Kotzig \cite{K}. Kotzig proved that all the Euler circuits of a
4-regular graph are obtained from any one using $\kappa$%
\emph{-transformations}. If a 4-regular graph is directed in such a way that
every vertex has indegree 2 and outdegree 2, then Kotzig \cite{K}, Pevzner
\cite{Pev} and Ukkonen \cite{U} showed that all of the graph's directed Euler
circuits are obtained from any one through certain combinations of $\kappa
$-transformations called \emph{transpositions} by Arratia, Bollob\'{a}s and
Sorkin \cite{A1, A2, A}. Bouchet \cite{Bold} and Rosenstiehl and Read
\cite{RR} introduced a simple graph associated with any\ Euler circuit of a
connected 4-regular graph, the \emph{alternance} graph or \emph{interlacement}
graph; an equivalent \emph{link relation matrix} was defined by\ Cohn and
Lempel \cite{CL} in the context of the theory of permutations. These authors
showed that the effects of $\kappa$-transformations and transpositions on
interlacement graphs are given by \emph{simple} local complementations and
edge pivots, respectively.

In the late 1980s, Bouchet introduced two new kinds of combinatorial
structures associated with these two theories. On the one hand are the
\emph{delta-matroids }\cite{Bdm}, some of which are associated with looped
simple graphs. The fundamental operation of delta-matroid theory is a way of
changing one delta-matroid into another, called \emph{twisting}. Two looped
simple graphs are related through PPT operations if and only if their
associated delta-matroids are related through twisting. On the other hand are
the \emph{isotropic systems} \cite{Bi1, Bi2}, all of which are associated with
\emph{fundamental graphs}. Two isotropic systems are \emph{strongly
isomorphic} if and only if they share fundamental graphs. Moreover, two simple
graphs are related through simple local complementations if and only if they
are fundamental graphs of strongly isomorphic isotropic systems. Properties of
isotropic systems were featured in the proof of Bouchet's famous
\textquotedblleft forbidden minors\textquotedblright\ characterization of
circle graphs \cite{Bco}.

The purpose of this paper is to introduce a binary matroid constructed in a
natural way from the adjacency matrix of a looped simple graph $G$; we call it
the \emph{isotropic matroid of} $G$, in honor of Bouchet's isotropic systems.
Let $G$ be a looped simple graph with adjacency matrix $A(G)$. That is, $A(G)$
is the $\left\vert V(G)\right\vert \times\left\vert V(G)\right\vert $ matrix
with entries in $GF(2)$ given by: a diagonal entry is $1$ if and only if the
corresponding vertex is looped, and an off-diagonal entry is $1$ if and only
if the corresponding vertices are adjacent. Let $IAS(G)$ denote the
$\left\vert V(G)\right\vert \times(3\left\vert V(G)\right\vert )$ matrix%
\[
IAS(G)=(I\mid A(G)\mid I+A(G)).
\]

\begin{definition}
The isotropic matroid of $G$ is the binary matroid $M[IAS(G)]$ represented by
$IAS(G)$.
\end{definition}

Let $W(G)$ denote the ground set of $M[IAS(G)]$, i.e., the set of columns of
$IAS(G)$. If $v\in V(G)$ then there are three columns of $IAS(G)$
corresponding to $v$: one in $I$, one in $A(G)$, and one in $I+A(G)$. For
notational convenience, and to indicate the connection with our work on
interlace polynomials \cite{T6, T7, Tra}, we use $v_{\phi}$ to denote the
column of $I$ corresponding to $v$, $v_{\chi}$ to denote the column of $A(G)$
corresponding to $v$, and $v_{\psi}$ to denote the column of $I+A(G)$
corresponding to $v$. The set $\{v_{\phi},v_{\chi},v_{\psi}\}$ is the
\emph{vertex triple} corresponding to $v$.

Notice that if $G_{2}$ is obtained from $G_{1}$ by loop complementation at a
vertex $v$ then there is an isomorphism between the isotropic matroids
$M[IAS(G_{1})]$ and $M[IAS(G_{2})]$ that simply interchanges the $v_{\chi}$
and $v_{\psi}$ elements of $W(G_{1})$ and $W(G_{2})$. We say isomorphisms like
this, which map vertex triples to vertex triples, are \emph{compatible with
the partitions of} $W(G_{1})$ \emph{and} $W(G_{2})$ \emph{into vertex
triples}, or simply \emph{compatible}. In Section 4 we observe that edge
pivots and local complementations also induce compatible isomorphisms of
isotropic matroids. Moreover, every compatible isomorphism is induced by some
sequence of edge pivots, local complementations and loop complementations. It
follows that compatible isomorphisms of\ isotropic matroids classify simple
graphs and looped simple graphs under various combinations of these
operations. For instance:

\begin{theorem}
\label{lcintro2}Let $G_{1}$ and $G_{2}$ be simple graphs. Then the following
conditions are equivalent:

\begin{enumerate}
\item Up to isomorphism, $G_{2}$ can be obtained from $G_{1}$ using simple
local complementations.

\item There is a compatible isomorphism $M[IAS(G_{1})]\cong M[IAS(G_{2})]$.
\end{enumerate}
\end{theorem}

\begin{theorem}
\label{lcintro}Let $G_{1}$ and $G_{2}$ be looped simple graphs. Then the
following conditions are equivalent:

\begin{enumerate}
\item Up to isomorphism, $G_{2}$ can be obtained from $G_{1}$ using local
complementations and loop complementations.

\item There is a compatible isomorphism $M[IAS(G_{1})]\cong M[IAS(G_{2})]$.
\end{enumerate}
\end{theorem}

\begin{theorem}
\label{pptintro}Let $G_{1}$ and $G_{2}$ be simple graphs. Then the following
conditions are equivalent:

\begin{enumerate}
\item Up to isomorphism, $G_{2}$ can be obtained from $G_{1}$ using edge pivots.

\item There is a compatible isomorphism $M[IAS(G_{1})]\cong M[IAS(G_{2})]$,
which maps $\psi$ elements to $\psi$ elements.
\end{enumerate}
\end{theorem}

\begin{theorem}
\label{pptintro2}Let $G_{1}$ and $G_{2}$ be looped simple graphs. Then the
following conditions are equivalent:

\begin{enumerate}
\item Up to isomorphism, $G_{2}$ can be obtained from $G_{1}$ using PPT operations.

\item There is a compatible isomorphism $M[IAS(G_{1})]\cong M[IAS(G_{2})]$,
which maps $\psi$ elements to $\psi$ elements.
\end{enumerate}
\end{theorem}

These results raise a natural question: what is the significance of
non-compatible isomorphisms between isotropic matroids? This question is
answered in Section 5, where we show that an arbitrary isomorphism between
isotropic matroids must yield a compatible isomorphism. We deduce the
following strengthenings of Theorems \ref{lcintro2} and \ref{lcintro}:

\begin{theorem}
\label{lcintro4}Let $G_{1}$ and $G_{2}$ be simple graphs. Then $M[IAS(G_{1}%
)]\cong M[IAS(G_{2})]$ if and only if up to isomorphism, $G_{2}$ can be
obtained from $G_{1}$ using simple local complementations.
\end{theorem}

\begin{theorem}
\label{lcintro3}Let $G_{1}$ and $G_{2}$ be looped simple graphs. Then
$M[IAS(G_{1})]\cong M[IAS(G_{2})]$ if and only if up to isomorphism, $G_{2}$
can be obtained from $G_{1}$ using local complementations and loop complementations.
\end{theorem}

Theorems \ref{lcintro2} -- \ref{lcintro3} tell us indirectly that the
isotropic matroid of a looped simple graph determines the graph's isotropic
system and the twist class of the graph's delta-matroid. In Section 6 we show
how to obtain the delta-matroid and isotropic system of $G$ directly from
$M[IAS(G)]$.

In Section 7 we discuss some fundamental properties of isotropic matroids. For
instance, $M[IAS(G)]$ is a connected matroid if and only if $G$ is a connected
graph with at least two vertices, and $M[IAS(G)]$ is a regular matroid if and
only if no connected component of $G$ contains more than two vertices.

In the last section we show that $M[IAS(G)]$ has another interesting property:
appropriately parametrized Tutte polynomials of isotropic matroids yield the
interlace polynomials introduced by Arratia, Bollob\'{a}s and Sorkin \cite{A1,
A2, A}, and also the modified versions subsequently defined by\ Aigner and van
der Holst \cite{AH}, Courcelle \cite{C} and the author \cite{T6}.

The ideas in this paper came to mind after the resemblance between the
matrices appearing in Aigner and van der Holst's discussion of interlace
polynomials \cite{AH} and our nonsymmetric approach to interlacement in
4-regular graphs \cite{T7} was pointed out to us by Robert Brijder. We are
grateful to him for years of informative correspondence regarding
delta-matroids, isotropic systems, PPT and related combinatorial notions. We
are also grateful to two anonymous readers, whose advice improved the original
version of the paper.

\section{Standard representations of binary matroids}

We do not review general results and terminology of graph theory and matroid
theory here; instead we refer the reader to standard texts in the field,
\cite{GM, O, We, W1} for instance. All the matroids we consider in this paper
are \emph{binary}:

\begin{definition}
Let $S$ be a finite set. A \emph{binary matroid} $M$ on $S$ is represented by
a matrix with entries in $GF(2)$, whose columns are indexed by the elements of
$S$. A subset of $S$ is dependent in $M$ if and only if the corresponding
columns of the matrix are linearly dependent.
\end{definition}

The binary matroid represented by a matrix is not changed if one row is added
to another, or the rows are permuted, or a row of zeroes is adjoined or
removed. Also, permuting the columns of a matrix will yield a new matrix that
represents an isomorphic binary matroid. Familiar results of elementary linear
algebra tell us that consequently, every binary matroid has a representation
of the following type:

\begin{definition}
Let $I$ be an $r\times r$ identity matrix. A \emph{standard representation }of
a rank-$r$ binary matroid $M$ is a matrix of the form $(I\mid A)$ that
represents $M$.
\end{definition}

If $A$ is a matrix with entries in $GF(2)$ then $M[IA]$ denotes the matroid
with standard representation $(I\mid A)$.

Recall that if $B$ is a basis of a matroid $M$, and $x$ is an element of $M$
not included in $B$, then the \emph{fundamental circuit} of $x$ with respect
to $B$ is%
\[
C(x,B)=\{x\}\cup\{b\in B\mid B\Delta\{b,x\}\text{ is a basis of }M\}\text{,}%
\]
where $\Delta$ denotes the symmetric difference. $C(x,B)$ is the unique
circuit contained in $B\cup\{x\}$.

A peculiar property of binary matroids is that the fundamental circuits with
respect to any one basis contain enough information to determine a binary
matroid. The same is not true for general matroids; for instance a matroid on
$\{1,2,3,4\}$ with basis $\{1,2\}$ and fundamental circuits $\{1,2,3\}$ and
$\{1,2,4\}$ might be either $U_{2,4}$ or the circuit matroid of a triangle
with one doubled edge. ($U_{2,4}$ is not binary, of course.) Notice that in
essence, a standard representation $(I\mid A)$ is this kind of description:
the matroid elements corresponding to the columns of $I$ constitute a basis
$B$, and for each element $x\notin B$, the fundamental circuit $C(x,B)$
includes $x$ together with the elements of $B$ corresponding to nonzero
entries of the $x$ column of $A$.

The only part of this section that does not appear in the textbooks mentioned
above is the following simple theorem, which tells us how the various standard
representations of a binary matroid are related to each other.

\begin{theorem}
\label{fund}Let $A_{1}$ and $A_{2}$ be $r\times(n-r)$ matrices with entries in
$GF(2)$. Then $M[IA_{1}]\cong M[IA_{2}]$ if and only if $(I\mid A_{2})$ can be
obtained from $(I\mid A_{1})$ using the following three types of operations on
matrices of the form $(I\mid A)$:

(a) Permute the columns of $A$.

(b) Permute the columns of $I$ and the rows of $(I\mid A)$, using the same permutation.

(c) Suppose the $jk$ entry of $A$ is $a_{jk}=1$. Then replace $a_{bc}$ with
$1+a_{bc}$ whenever $b\neq j$, $c\neq k$, $a_{jc}=1$ and $a_{bk}=1$.
\end{theorem}

\begin{proof}
As noted above, a standard presentation of a rank-$r$ binary matroid $M$ on an
$n$-element set $S$ is obtained as follows.\ First choose a basis $B$, and
index its elements as $s_{1},...,s_{r}$. Then index the remaining elements of
$S$ as $s_{r+1},...,s_{n}$. Finally, let $A$ be the $r\times(n-r)$ matrix
whose $jk$ entry is 1 if and only if $s_{j}$ is an element of the fundamental
circuit $C(s_{r+k},B)$.

Operations of types (a) and (b) correspond to re-indexings of $S-B$ and $B$, respectively.

Suppose now that $a_{jk}=1$, and let $A^{\prime}$ be the matrix obtained from
$A$ by an operation of type (c). Another way to describe $A^{\prime}$ is this:
$(I\mid A^{\prime})$ is obtained from $(I\mid A)$ by first interchanging the
$j$th and $(r+k)$th columns, and then adding the $j$th row of the resulting
matrix to every other row in which the original $(r+k)$th column has a nonzero
entry. That is, the matrix $(I\mid A^{\prime})$ is simply the standard
representation corresponding to the basis $B\Delta\{s_{j},s_{r+k}\}$, with the
elements other than $s_{j}$ and $s_{r+k}$ indexed as they were before.

The theorem follows, because basis exchanges $B\mapsto B\Delta\{b,x\}$
eventually construct every basis of $M$ from any one.
\end{proof}

We refer to an operation of type (c) as a \emph{basis exchange} involving the
$j$th column of $I$ and the $k$th column of $A$. (It would also be natural to
call it a \emph{pivot}, but this term already has other meanings.)

\section{$M[IAS(G)]$ and compatible isomorphisms}

If $G$ is a looped simple graph then $A(G)$ denotes the adjacency matrix of
$G$, and $AS(G)$ denotes the matrix $(A(G)\mid I+A(G))$. (\thinspace$S$ is for
\textquotedblleft sum.\textquotedblright) As mentioned in the introduction,
the ground set $W(G)$ of the isotropic matroid $M[IAS(G)]$ is partitioned into
three-element vertex triples; the vertex $v$ corresponds to the vertex triple
$\{v_{\phi},v_{\chi},v_{\psi}\}$.

It is convenient to adopt notation to describe matroid isomorphisms that are
compatible with these vertex triples. Let $S_{3}$ denote the group of
permutations of the three symbols $\phi$, $\chi$ and $\psi$. We use standard
notation in $S_{3}$: for instance $1$ is the identity, $(\phi\chi)$ is a
transposition, and $(\phi\chi)(\chi\psi)=(\psi\phi\chi)$ is a 3-cycle.

Suppose $G_{1}$ and $G_{2}$ are looped simple graphs, and there is a
compatible isomorphism $\beta:M[IAS(G_{1})]\rightarrow M[IAS(G_{2})]$. Then
the isomorphism consists of two parts. First, there is an induced bijection
$V(G_{1})\rightarrow V(G_{2})$; in general we will denote this bijection
$\beta$ too, though up to isomorphism we may always presume that
$V(G_{1})=V(G_{2})$ and the induced bijection is the identity map. Second,
there is a function $f_{\beta}:V(G_{1})\rightarrow S_{3}$ such that
$\beta(v_{\iota})=\beta(v)_{f_{\beta}(v)(\iota)}$ $\forall v\in V(G_{1})$
$\forall\iota\in\{\phi,\chi,\psi\}$. In this situation we say that $\beta$
is\emph{ determined} \emph{by} $f_{\beta}$.

Here are two obvious properties of compatible isomorphisms.

\begin{lemma}
\label{product}If $\beta_{1}:M[IAS(G_{1})]\rightarrow M[IAS(G_{2})]$ and
$\beta_{2}:M[IAS(G_{2})]$ $\rightarrow M[IAS(G_{3})]$ are compatible
isomorphisms then so is $\beta_{2}\circ\beta_{1}:M[IAS(G_{1})]\rightarrow$
$M[IAS(G_{3})]$, and it is determined by the map $f:V(G_{1})\rightarrow S_{3}$
given by $f(v)=f_{\beta_{2}}(\beta_{1}(v))\cdot f_{\beta_{1}}(v)$.
\end{lemma}

\begin{lemma}
\label{inverse}If $\beta:M[IAS(G_{1})]\rightarrow M[IAS(G_{2})]$ is a
compatible isomorphism then so is $\beta^{-1}:M[IAS(G_{2})]\rightarrow
M[IAS(G_{1})]$, and $\beta^{-1}$ is determined by the map $f_{\beta^{-1}%
}:V(G_{2})\rightarrow S_{3}$ given by $f_{\beta^{-1}}(v)=f_{\beta}(\beta
^{-1}(v))^{-1}$.
\end{lemma}

The next property is not quite so obvious.

\begin{lemma}
\label{identity}Suppose $\beta:M[IAS(G_{1})]$ $\rightarrow M[IAS(G_{2})]$ is a
compatible isomorphism, determined by the map $f:V(G_{1})\rightarrow S_{3}$
with $f(v)=1$ $\forall v\in V(G_{1})$. Then the bijection $\beta
:V(G_{1})\rightarrow V(G_{2})$ is an isomorphism between $G_{1}$ and $G_{2}$.
\end{lemma}

\begin{proof}
Up to isomorphism, we may as well presume that $V(G_{1})=V(G_{2})$ and the
bijection $V(G_{1})\rightarrow V(G_{2})$ induced by $\beta$ is the identity
map. Then $M[IAS(G_{1})]$ and $M[IAS(G_{2})]$ are isomorphic matroids on the
ground set $W(G_{1})$ $=W(G_{2})$. As $f(v)\equiv1$, the compatible
isomorphism $\beta$ is the identity map of this ground set.

The identity map preserves the basis $\Phi=\{v_{\phi}\mid v\in V(G_{1})\}$.
The identity map is a matroid isomorphism, so it must also preserve
fundamental circuits with respect to $\Phi$. Recall the discussion of Section
2: the column of $IAS(G_{i})$ corresponding to $x\notin\Phi$ is determined by
the fundamental circuit of $x$ with respect to $\Phi$ in $M[IAS(G_{i})]$. It
follows that the matrices $AS(G_{1})$ and $AS(G_{2})$ are identical.
\end{proof}

Lemmas \ref{product} and \ref{identity} imply the following.

\begin{corollary}
\label{uniqueness}Suppose $\beta_{1}:M[IAS(G_{1})]\rightarrow M[IAS(G_{2})]$
and $\beta_{2}:M[IAS(G_{1})]$ $\rightarrow M[IAS(G_{3})]$ are compatible
isomorphisms, and the associated functions $f_{\beta_{1}}$, $f_{\beta_{2}%
}:V(G_{1})\rightarrow S_{3}$ are the same. Then the bijection $V(G_{2}%
)\rightarrow V(G_{3})$ defined by $\beta_{2}\circ\beta_{1}^{-1}$ is an
isomorphism between $G_{2}$ and $G_{3}$.
\end{corollary}

\section{Complements and pivots}

In this section we prove that the matroid $M[IAS(G)]$ classifies $G$ under
several different kinds of operations.

\subsection{Loop complementation}

Suppose $G_{1}$ is a looped simple graph, $v\in V(G_{1})$, and $G_{2}$ is the
graph obtained from $G_{1}$ by complementing (reversing) the loop status of
$v$. Clearly then $IAS(G_{2})$ is the matrix obtained from $IAS(G_{1})$ by
interchanging the $v_{\chi}$ and $v_{\psi}$ columns. This interchange is an
example of an operation of type (a), so Theorem \ref{fund} tells us that there
is a compatible isomorphism $M[IAS(G_{1})]\rightarrow M[IAS(G_{2})]$
determined by the map $f:V(G_{1})\rightarrow S_{3}$ given by%
\[
f(w)=\left\{
\begin{array}
[c]{cc}%
\text{the transposition }(\chi\psi)\text{,} & \text{if }w=v\\
& \\
1\text{,} & \text{if }w\neq v
\end{array}
\right.  .
\]

The converse also holds:

\begin{theorem}
\label{thmloop}Let $G_{1}$ and $G_{2}$ be looped simple graphs, and suppose
$v\in V(G_{1})$. Then these two conditions are equivalent:

\begin{enumerate}
\item Up to isomorphism, $G_{2}$ is the graph obtained from $G_{1}$ by
complementing the loop status of $v$.

\item There is a compatible isomorphism $\beta:M[IAS(G_{1})]\rightarrow
M[IAS(G_{2})]$ such that $f_{\beta}(v)=(\chi\psi)$ and $f_{\beta}(w)=1$
$\forall w\neq v$.
\end{enumerate}
\end{theorem}

\begin{proof}
We have already discussed the implication 1$\Rightarrow$2. The converse
follows from Corollary \ref{uniqueness}.
\end{proof}

\subsection{Local complementation}

Two different versions of \emph{local complementation} appear in the
literature. \emph{Simple} local complementation was introduced by Bouchet
\cite{Bold} and Rosenstiehl and Read \cite{RR}, as part of the theory of
interlacement in 4-regular graphs. This operation does not involve the
creation of loops, so it is the version seen most often in graph theory, where
the theory of simple graphs predominates. \emph{Non-simple} local
complementation is part of the theory of the principal pivot transform (PPT)
over $GF(2)$. The general theory of PPT was introduced by Tucker \cite{Tu};
see also the survey of Tsatsomeros \cite{Ts}. The special significance of
non-simple local complementation in PPT over $GF(2)$ was discussed by Geelen
\cite{G}. Later (and independently) non-simple local complementation was
introduced by Arratia, Bollob\'{a}s and Sorkin as part of the theory of the
two-variable interlace polynomial \cite{A}. We should emphasize that simple
local complementations are usually applied only to simple graphs in the first
set of references, and non-simple local complementations are usually applied
only with respect to looped vertices in the second set of references. Our
definitions are not so restrictive.

\begin{definition}
If $G$ is a looped simple graph and $v\in V(G)$ then the \emph{simple local
complement of }$G$ \emph{with respect to }$v$ is the graph $G_{s}^{v}$
obtained from $G$ by complementing all adjacencies between distinct elements
of the open neighborhood $N(v)$. The \emph{non-simple local complement of }$G$
\emph{with respect to }$v$ is the graph $G_{ns}^{v}$ obtained from $G_{s}^{v}$
by complementing the loop status of each element of $N(v)$.
\end{definition}

Observe that replacing $A(G)$ by $A(G_{ns}^{v})$ has precisely the same effect
on the matrix $IAS(G)$ as a type (c) operation from Theorem \ref{fund}. As
discussed in Section 2, this operation is equivalent to a basis exchange
involving $v_{\phi}$ and either $v_{\chi}$ (if $v$ is looped) or $v_{\psi}$
(if $v$ is unlooped). We deduce the following.

\begin{theorem}
\label{lcinv}Let $G_{1}$ and $G_{2}$ be looped simple graphs, and suppose
$v\in V(G_{1})$. Then these two conditions are equivalent:

\begin{enumerate}
\item \noindent$G_{2}$ is isomorphic to $(G_{1})_{ns}^{v}$.

\item \noindent There is a compatible isomorphism $\beta:M[IAS(G_{1}%
)]\rightarrow M[IAS(G_{2})]$ such that $f_{\beta}(w)=1$ $\forall w\neq v$ and
\[
f_{\beta}(v)=\left\{
\begin{array}
[c]{cc}%
(\phi\chi)\text{,} & \text{if }v\text{ is looped in }G_{1}\\
& \\
(\phi\psi) & \text{if }v\text{ is not looped in }G_{1}%
\end{array}
\right.  .
\]
\noindent
\end{enumerate}
\end{theorem}

\begin{proof}
As already noted, 1$\Rightarrow$2 because $IAS((G_{1})_{ns}^{v})$ is the same
as the matrix associated with the standard presentation of $M[IAS(G_{1})]$
obtained from $IAS(G_{1})$ by a basis exchange involving $v_{\phi}$ and either
$v_{\chi}$ or $v_{\psi}$. The converse follows from Corollary \ref{uniqueness}.
\end{proof}

\subsection{Pivots}

Here is a well-known definition. The reader who is encountering it for the
first time should take a moment to verify that the two indicated triple local
complements are indeed the same.

\begin{definition}
\label{pivdef}If $v$ and $w$ are neighbors in $G$ then the \emph{edge pivot}
$G^{vw}$ is the triple simple local complement:%
\[
G^{vw}=((G_{s}^{v})_{s}^{w})_{s}^{v}=((G_{s}^{w})_{s}^{v})_{s}^{w}.
\]

\end{definition}

Note that we do not restrict edge pivots to edges with unlooped vertices.

\begin{corollary}
\label{pivinv}Let $G_{1}$ and $G_{2}$ be looped simple graphs, and suppose
$v\neq w$ are neighbors in $G_{1}$. Then these two conditions are equivalent:

\begin{enumerate}
\item $G_{2}$ is isomorphic to $(G_{1})^{vw}$.

\item There is a compatible isomorphism $\beta:M[IAS(G_{1})]\rightarrow
M[IAS(G_{2})]$ such that%
\[
f_{\beta}(v)=\left\{
\begin{array}
[c]{cc}%
(\phi\chi)\text{,} & \text{if }v\text{ is unlooped}\\
& \\
(\phi\psi)\text{,} & \text{ if }v\text{ is looped}%
\end{array}
\right.  \text{, }f_{\beta}(w)=\left\{
\begin{array}
[c]{cc}%
(\phi\chi)\text{,} & \text{if }w\text{ is unlooped}\\
& \\
(\phi\psi)\text{,} & \text{ if }w\text{ is looped}%
\end{array}
\right.
\]
and $f_{\beta}(x)=1$ $\forall x\notin\{v,w\}$.
\end{enumerate}
\end{corollary}

\begin{proof}
The reader can easily check that the definition%
\[
G_{1}^{vw}=(((G_{1})_{s}^{v})_{s}^{w})_{s}^{v}%
\]
is equivalent to saying this: $G_{1}^{vw}$ is obtained from $(((G_{1}%
)_{ns}^{v})_{ns}^{w})_{ns}^{v}$ by complementing the loop status of $v$. (The
reason is that $v$ is the only vertex whose loop status is complemented an odd
number of times in $(((G_{1})_{ns}^{v})_{ns}^{w})_{ns}^{v}$.) It follows that
there is a compatible isomorphism $M[IAS(G_{1})]\rightarrow M[IAS(G_{1}%
^{vw})]$ obtained by composing four compatible isomorphisms, three from
Theorem \ref{lcinv} and one from Theorem \ref{thmloop}. According to Lemma
\ref{product}, this compatible isomorphism is determined by the function
$f:V(G)\rightarrow S_{3}$ such that $f(x)=1$ for $w\neq x\neq v$,
$f(w)=(\phi\psi)$ if $w$ is unlooped in $(G_{1})_{ns}^{v}$, $f(w)=(\phi\chi)$
if $w$ is looped in $(G_{1})_{ns}^{v}$, and
\[
f(v)=\left\{
\begin{array}
[c]{cc}%
(\chi\psi)\cdot(\phi\chi)\cdot1\cdot(\phi\psi)\text{,} & \text{if }v\text{ is
unlooped}\\
& \\
(\chi\psi)\cdot(\phi\psi)\cdot1\cdot(\phi\chi)\text{,} & \text{ if }v\text{ is
looped}%
\end{array}
\right.  \text{ .}%
\]

This verifies the implication 1$\Rightarrow$2. The converse follows from
Corollary \ref{uniqueness}.
\end{proof}

Notice that the function $f_{\beta}$ of Corollary \ref{pivinv} is the
combination of two separate functions, one $\equiv1$ except at $v$ and the
other $\equiv1$ except at $w$. According to Theorem \ref{fund}, these two
separate functions do not come from two separate compatible isomorphisms,
though. This fact is reflected in the proof, where the compatible isomorphism
$M[IAS(G_{1})]\rightarrow M[IAS(G_{1}^{vw})]$ is described as a composition of
\emph{four} simpler compatible isomorphisms, not two.

There is a different way to describe the compatible isomorphism of Corollary
\ref{pivinv}, using only two basis exchanges. According to Theorem \ref{fund},
if $v$ and $w$ are neighbors then there is a basis exchange involving
$v_{\phi}$ and either $w_{\chi}$ or $w_{\psi}$, as each of these columns has a
$1$ in the $v$ row. The matrix resulting from part (c) of Theorem \ref{fund}
is not of the form $(I\mid A\mid I+A)$ for a symmetric matrix $A$, so there is
no natural way to interpret such a basis exchange as a graph operation.
However, if this basis exchange is followed by one involving $w_{\phi}$ and
either $v_{\chi}$ or $v_{\psi}$, then the result \emph{is} of the form $(I\mid
A\mid I+A)$ for a symmetric matrix $A$. (We leave details to the interested
reader.) Moreover, $A$ closely resembles the adjacency matrix of $G_{1}^{vw}$;
the positions of $v$ and $w$ have been interchanged, though, and depending on
the $\chi,\psi$ choices the loop statuses of $v$ and $w$ may also have
changed. The fact that the compatible isomorphism $M[IAS(G_{1})]\rightarrow
M[IAS(G_{1}^{vw})]$ may be described in this different way, involving a
transposition of adjacency information regarding $v$ and $w$, is a reflection
of the fact that there is a different way to define the pivot. See Section 3
of \cite{A2}, where Arratia, Bollob\'{a}s and Sorkin give this different
definition, and show that it is related to Definition \ref{pivdef} by applying
a \textquotedblleft label swap\textquotedblright\ exchanging the names of $v$
and $w$.

\subsection{Classifying graphs using compatible isomorphisms}

The simplest classification theorem resulting from the above discussion is
this immediate consequence of Theorem \ref{thmloop}.

\begin{theorem}
Let $G_{1}$ and $G_{2}$ be looped simple graphs. Then these two conditions are equivalent:

\begin{enumerate}
\item Up to isomorphism, $G_{2}$ can be obtained from $G_{1}$ by complementing
the loop status of some vertices.

\item There is a compatible isomorphism $\beta:M[IAS(G_{1})]\rightarrow
M[IAS(G_{2})]$ such that $f_{\beta}(v)\in\{1,(\chi\psi)\}$ $\forall v\in
V(G_{1})$.
\end{enumerate}
\end{theorem}

Other classification theorems have similar statements, but take a little more
work to prove.

\begin{theorem}
\label{thmpiv}Let $G_{1}$ and $G_{2}$ be looped simple graphs. Then these two
conditions are equivalent:

\begin{enumerate}
\item Up to isomorphism, $G_{2}$ can be obtained from $G_{1}$ using edge pivots.

\item There is a compatible isomorphism $\beta:M[IAS(G_{1})]\rightarrow
M[IAS(G_{2})]$ such that $f_{\beta}(v)\in\{1,(\phi\psi)\}$ for every looped
$v\in V(G_{1})$ and $f_{\beta}(v)\in\{1,(\phi\chi)\}$ for every unlooped $v\in
V(G_{1})$.
\end{enumerate}
\end{theorem}

\begin{proof}
If $G^{\prime}$ can be obtained from $G$ using edge pivots, simply apply
Corollary \ref{pivinv} repeatedly.

For the converse, suppose condition 2 holds, and there are $k$ vertices with
$f_{\beta}(v)\neq1$. If $k=0$ then Corollary \ref{uniqueness} implies that
$G_{1}\cong G_{2}$.

If $k>0$ then let $v_{0}\in V(G)$ have $f_{\beta}(v_{0})\neq1$, and let
$\Phi=\{v_{\phi}\mid v\in V(G_{1})\}$. Then $\Phi$ is a basis of
$M[IAS(G_{1})]$, so $\beta(\Phi)$ is a basis of $M[IAS(G_{2})]$.
Consequently,
\[
\{\beta(v_{0\phi})\}\cup\{v_{\phi}\mid f_{\beta}(v_{\phi})=1\}
\]
cannot be dependent in $M[IAS(G_{2})]$, because it is a subset of $\beta
(\Phi)$. It follows that the column of $IAS(G_{2})$ corresponding to
$\beta(v_{0\phi})$ must have at least one nonzero entry in a row that
corresponds to a vertex $v\neq v_{0}$ with $f_{\beta}(v)\neq1$. Then $v$ is a
neighbor of $v_{0}$, and Corollary \ref{pivinv} implies that there is a
compatible isomorphism $\beta^{\prime}:M[IAS(G_{2})]\rightarrow M[IAS(G_{2}%
^{vv_{0}})]$ determined by the function $f_{\beta^{\prime}}:V(G_{2}%
)\rightarrow S_{3}$ with $f_{\beta^{\prime}}(w)=1$ $\forall w\notin
\{v,v_{0}\}$, $f_{\beta^{\prime}}(v_{0})=f_{\beta}(v_{0})$ and $f_{\beta
^{\prime}}(v)=f_{\beta}(v)$. The composition $\beta^{\prime}\circ\beta$ is a
compatible isomorphism $M[IAS(G_{1})]\rightarrow M[IAS(G_{2}^{vv_{0}})]$,
which satisfies condition 2; as $(f_{\beta^{\prime}\circ\beta})^{-1}%
(1)=\{v,v_{0}\}\cup f_{\beta}^{-1}(1)$, induction assures us that up to
isomorphism, $G_{2}^{vv_{0}}$ may be obtained from $G_{1}$ using edge pivots.
\end{proof}

Restricting Theorem \ref{thmpiv} to simple graphs yields Theorem
\ref{pptintro} of the introduction. The following corollary of Theorem
\ref{thmpiv} implies Theorem \ref{pptintro2}:

\begin{corollary}
\label{pptinv}Let $G_{1}$ and $G_{2}$ be looped simple graphs. Then these
three conditions are equivalent:

\begin{enumerate}
\item Up to isomorphism, $G_{2}$ can be obtained from $G_{1}$ using two kinds
of operations: non-simple local complementations with respect to looped
vertices, and edge pivots with respect to non-looped vertices.

\item Up to isomorphism, $G_{2}$ can be obtained from $G_{1}$ using PPT operations.

\item There is a compatible isomorphism $\beta:M[IAS(G_{1})]\rightarrow
M[IAS(G_{2})]$ such that $f_{\beta}(v)\in\{1,(\phi\chi)\}$ $\forall v\in
V(G_{1})$.
\end{enumerate}
\end{corollary}

\begin{proof}
The equivalence 1$\Leftrightarrow$2 is due to Geelen \cite{G}, and the
implication 1$\Rightarrow$3 follows from Theorems \ref{lcinv} and \ref{thmpiv}.

Suppose condition 3 holds and there are $k$ vertices with $f_{\beta}(v)\neq1$.
If $k=0$ then Corollary \ref{uniqueness} implies that $G_{1}$ and $G_{2}$ are
isomorphic. If $G_{1}$ has a looped vertex $v_{0}$ with $f_{\beta}(v_{0}%
)\neq1$ then there is a compatible isomorphism $\beta^{\prime}:M[IAS((G_{1}%
)_{ns}^{v_{0}})]\rightarrow M[IAS(G_{2})]$ that satisfies condition 3, and for
which only $k-1$ vertices have $f_{\beta^{\prime}}(v)\neq1$. Induction then
tells us that condition 1 holds. If there is no such $v_{0}$, then Theorem
\ref{thmpiv} applies.
\end{proof}

It remains to prove Theorem \ref{lcintro} of the introduction. Let $G_{1}$ and
$G_{2}$ be looped simple graphs. If $G_{2}$ can be obtained from $G_{1}$ using
local complementations and loop complementations then Theorems \ref{thmloop}
and \ref{lcinv} tell us that there is a compatible isomorphism between their
isotropic matroids.

For the converse, suppose there is a compatible isomorphism $\beta
:M[IAS(G_{1})]$ $\rightarrow M[IAS(G_{2})]$, and there are $k$ vertices with
$f_{\beta}(v)\neq1$. Up to isomorphism, we may presume that $V(G_{1}%
)=V(G_{2})$ and the bijection induced by $\beta$ is the identity map. If $k=0$
then Corollary \ref{uniqueness} tells us that $G_{1}=G_{2}$.

The argument proceeds by induction on $k\geq1$. If there is any vertex $v_{0}$
with $f_{\beta}(v_{0})=(\chi\psi)$ then by Theorem \ref{thmloop}, the graph
$G_{2}^{\prime}$ obtained from $G_{2}$ by complementing the loop status of
$v_{0}$ has the property that there is a compatible isomorphism $\beta
^{\prime}:M[IAS(G_{1})]\rightarrow M[IAS(G_{2}^{\prime})]$ such that
$f_{\beta^{\prime}}(v)=f_{\beta}(v)$ $\forall v\neq v_{0}$, and $f_{\beta
^{\prime}}(v_{0})=1$. The inductive hypothesis tells us that up to
isomorphism, $G_{2}^{\prime}$ can be obtained from $G_{1}$ using local
complementations and loop complementations. Of course we can then obtain
$G_{2}$ from $G_{2}^{\prime}$ by loop complementation.

If there is a looped vertex $v_{0}$ with $f_{\beta}(v_{0})=(\phi\chi)$ or an
unlooped vertex $v_{0}$ with $f_{\beta}(v_{0})=(\phi\psi)$, then a similar
argument applies, with $G_{2}^{\prime}=(G_{2})_{ns}^{v_{0}}$.

If there is a looped vertex $v_{0}$ with $f_{\beta}(v_{0})=(\phi\psi)$ or an
unlooped vertex $v_{0}$ with $f_{\beta}(v_{0})=(\phi\chi)$, then the same
argument used in the proof of Theorem \ref{thmpiv} tells us that there is a
vertex $v$ that neighbors $v_{0}$ in $G_{2}$ and has $f_{\beta}(v)\neq1$. Then
the graph $G_{2}^{\prime}=G_{2}^{v_{0}v}$ has the property that there is a
compatible isomorphism $\beta^{\prime}:M[IAS(G_{1})]\rightarrow M[IAS(G_{2}%
^{\prime})]$ such that $f_{\beta^{\prime}}(w)=f_{\beta}(w)$ $\forall
w\notin\{v_{0},v\}$, and $f_{\beta^{\prime}}(v_{0})=1$. The inductive
hypothesis tells us that up to isomorphism, $G_{2}^{\prime}$ can be obtained
from $G_{1}$ using local complementations and loop complementations; of course
we can then obtain $G_{2}$ from $G_{2}^{\prime}$ using local complementations.

Finally, if there is a vertex $v_{0}$ with $f_{\beta}(v_{0})$ a 3-cycle then
the graph $G_{2}^{\prime}$ obtained from $G_{2}$ by complementing the loop
status of $v_{0}$ has the property that there is a compatible isomorphism
$\beta^{\prime}:M[IAS(G_{1})]\rightarrow M[IAS(G_{2}^{\prime})]$ such that
$f_{\beta^{\prime}}(v)=f_{\beta}(v)$ $\forall v\neq v_{0}$, and $f_{\beta
^{\prime}}(v_{0})=(\chi\psi)\cdot f_{\beta}(v_{0})$ is a transposition.
Consequently one of the preceding arguments applies to $G_{2}^{\prime}$.

This completes the proof of Theorem \ref{lcintro}.

\subsection{$M[IA(G)]$}

If $G$ is a looped simple graph then we call the binary matroid $M[IA(G)]$ the
\emph{restricted} isotropic matroid of $G$; it is represented by the
$\left\vert V(G)\right\vert \times(2\left\vert V(G)\right\vert )$ matrix
$IA(G)=(I\mid A(G))$. This use of the term \emph{restricted} is consistent
with Bouchet's use of the term for isotropic systems \cite{B5}. (The
connection between isotropic matroids and isotropic systems is discussed
in\ Section 6.)

Suppose $G_{1}$ and $G_{2}$ are looped simple graphs, and there is a
compatible isomorphism $\beta:M[IAS(G_{1})]\rightarrow M[IAS(G_{2})]$ with
$f_{\beta}(v)\in\{1,(\phi\chi)\}$ $\forall v\in V(G_{1})$. Then $\beta
(v_{\psi})=\beta(v)_{\psi}$ $\forall v\in V(G_{1})$, so $\beta$ restricts to
an isomorphism between the submatroids $M[IA(G_{1})]$ and $M[IA(G_{2})]$.
Moreover, this restriction of $\beta$ is compatible with the natural
partitions of these matroids into pairs, and the restriction determines
$\beta$.

Theorem \ref{pptintro} implies that a simple graph is classified up to pivot
equivalence by compatible isomorphisms of $M[IA(G)]$. Similarly, Theorem
\ref{pptintro2} implies that a looped simple graph is classified up to PPT
equivalence by compatible isomorphisms of $M[IA(G)]$.

A compatible isomorphism $\beta:M[IAS(G_{1})]\rightarrow M[IAS(G_{2})]$ with
$f_{\beta}(v)\in\{1,(\phi\psi)\}$ $\forall v$ can be analyzed by first
applying loop complementation to all vertices in $G_{1}$ and $G_{2}$, and then
analyzing the corresponding compatible isomorphism $\beta^{\prime}%
:M[IAS(G_{1}^{\prime})]\rightarrow M[IAS(G_{2}^{\prime})]$, which has
$f_{\beta^{\prime}}(v)\in\{1,(\phi\chi)\}$ $\forall v\in V(G_{1}^{\prime})$.
Compatible isomorphisms with $f_{\beta}(v)\in\{1,(\chi\psi)\}$ $\forall v$ are
less interesting, as Theorem \ref{thmloop} tells us that they can be realized
using loop complementation.

\section{Non-compatible isomorphisms}

Recall that for each vertex $v$ of a looped simple graph $G$, the ground set
$W(G)$ of $M[IAS(G)]$ has a three-element vertex triple $\{v_{\phi},v_{\chi
},v_{\psi}\}$. The discussion of\ Section 4 relies on the fact that
transpositions of the symbols $\phi$, $\chi$, $\psi$ describe the effects on
these vertex triples of local complementations, loop complementations, and
edge pivots. If two graphs are not related by these graph operations then it
might seem possible for\ their isotropic matroids to be isomorphic, so long as
there is no isomorphism compatible with vertex triples. In fact, however, this
is impossible:

\begin{theorem}
\label{noncom}Let $G_{1}$ and $G_{2}$ be looped simple graphs. If there is an
isomorphism between $M[IAS(G_{1})]$ and $M[IAS(G_{2})]$, then there is a
compatible isomorphism between them.
\end{theorem}

\subsection{Triangulations of isotropic matroids}

We prove Theorem \ref{noncom} by carefully analyzing the images of the vertex
triples of $G_{1}$ under a non-compatible isomorphism $M[IAS(G_{1}%
)]\rightarrow M[IAS(G_{2})]$. These images satisfy the following.

\begin{definition}
Let $G$ be a looped simple graph. A partition $P$ of $W(G)$ into three-element
subsets is a \emph{triangulation} if each triple in $P$ contains either a
3-element circuit of $M[IAS(G)]$ or a loop and a pair of non-loop parallels.
\end{definition}

The partition of $W(G)$ into vertex triples is a triangulation, of course. We
call it the \emph{vertex triangulation}. The simplest non-vertex
triangulations of $W(G)$ are obtained from the vertex triangulation by
interchanging parallel elements of $M[IAS(G)]$ from distinct vertex triples.
It is not difficult to see that all such parallels in $M[IAS(G)]$ are
associated with isolated, pendant or twin vertices; we leave the details to
the reader.

Other non-vertex triangulations can be a little more complicated. Suppose $u$,
$v$, $w$ and $x$ are unlooped vertices in $G$ with $N(v)=\{u,w\}$,
$N(w)=\{v,x\}$ and $N(u)-\{v\}=N(x)-\{w\}$. We say $u$, $v$, $w$ and $x$
constitute a \emph{matched 4-path}. A non-vertex triangulation of $W(G)$ may
be obtained from the vertex triangulation by replacing the vertex triples
corresponding to $u$, $v$, $w$ and $x$ with these four triples: $\{u_{\phi
},v_{\chi},w_{\phi}\}$, $\{v_{\phi},w_{\chi},x_{\phi}\}$, $\{u_{\psi},v_{\psi
},x_{\chi}\}$ and $\{u_{\chi},w_{\psi},x_{\psi}\}$. We refer to this
replacement as \emph{bending} the 4-path.

Suppose $G^{\prime}$ is obtained from $G$ using local complementations and
loop complementations, and $u,v,w,x$ is a matched 4-path in $G$. Then we say
$u,v,w,x$ is a \emph{matched 4-set} in $G^{\prime}$. The terminology reflects
the fact that the subgraph of $G^{\prime}$ induced by a matched 4-set need not
be a path. For instance, a matched 4-path $u,v,w,x$ in $G$ yields a 4-cycle in
$G^{vw}$; see Figure \ref{lcmatfig}. The discussion of Section 4 tells us that
there is an isomorphism $\beta:M[IAS(G)]\rightarrow M[IAS(G^{\prime})]$ that
is compatible with the vertex triangulations, so a triangulation $P$ of $W(G)$
induces a triangulation $\beta(P)$ of $W(G^{\prime})$. In particular, if $P$
is a non-vertex triangulation of $W(G)$ in which $u,v,w,x$ is bent, then we
say that $u,v,w,x$ is a bent 4-set in $\beta(P)$.%

\begin{figure}
[ptb]
\begin{center}
\includegraphics[
trim=1.140271in 8.458904in 0.580332in 1.003711in,
height=1.1078in,
width=4.7703in
]%
{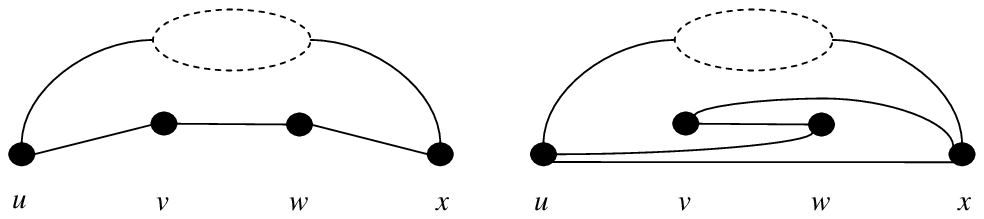}%
\caption{A matched 4-path in $G$ and the corresponding matched 4-set in
$G^{vw}$.}%
\label{lcmatfig}%
\end{center}
\end{figure}

\begin{proposition}
\label{triang}Let $G$ be a looped simple graph, and suppose $P$ is a
non-vertex triangulation of $W(G)$ obtained from the vertex triangulation
either by bending a matched 4-set in $G$ or by interchanging two parallel
elements of $M[IAS(G)]$. Then there is a matroid automorphism $\alpha
:M[IAS(G)]\rightarrow M[IAS(G)]$ such that $\alpha(P)$ is the vertex triangulation.
\end{proposition}

\begin{proof}
If $x$ and $y$ are parallel elements of a matroid, then the transposition
$(xy)$ is a matroid automorphism.

Suppose $u,v,w,x$ is a matched 4-path in $G$, and $P$ is obtained from the
vertex triangulation by bending the 4-path. Let $\alpha:W(G)\rightarrow W(G)$
be the permutation
\[
\alpha=(u_{\phi}x_{\phi})(u_{\chi}v_{\phi})(u_{\psi}w_{\chi})(v_{\chi}x_{\psi
})(v_{\psi}w_{\psi})(w_{\phi}x_{\chi}).
\]

As $\alpha$ is a bijection, to show that it defines an automorphism of the
matroid $M[IAS(G)]$ it suffices to verify that $\alpha$ is linear on the
columns of $IAS(G)$. As $\Phi=\{t_{\phi}\mid t\in V(G)\}$ is a basis of the
column space of $M[IAS(G)]$, we may verify linearity by checking that $\alpha$
is consistent with expressions of columns outside $\Phi$ as linear
combinations of elements of $\Phi$. This property is obvious for columns that
correspond to vertices outside $\{u,v,w,x\}$; these columns are fixed by
$\alpha$, the $u$ and $x$ entries in these columns are always equal, and
$\alpha$ transposes $u_{\phi}$ and $x_{\phi}$. The remaining columns outside
$\Phi$ are the $\chi$ and $\psi$ columns corresponding to $u$, $v$, $w$ and
$x$. It is a simple matter to verify the eight corresponding equalities
individually: if we let $N$ denote the column vector corresponding to
$N(u)-\{v\}=N(x)-\{w\}$ then $u_{\chi}=N+v_{\phi}$ and $\alpha(N)+\alpha
(v_{\phi})=N+u_{\chi}=N+(N+v_{\phi})=v_{\phi}=\alpha(u_{\chi})$; $u_{\psi
}=N+u_{\phi}+v_{\phi}$ and $\alpha(N)+\alpha(u_{\phi})+\alpha(v_{\phi
})=N+x_{\phi}+u_{\chi}=N+x_{\phi}+N+v_{\phi}=x_{\phi}+v_{\phi}=w_{\chi}%
=\alpha(u_{\psi})$; $v_{\chi}=u_{\phi}+w_{\phi}$ and $\alpha(u_{\phi}%
)+\alpha(w_{\phi})=x_{\phi}+x_{\chi}=x_{\psi}=\alpha(v_{\chi})$; etc.

As $\alpha(P)$ is the vertex triangulation, $\alpha$ satisfies the proposition.

Suppose now that $u,v,w,x$ is a matched 4-set in $G$, and $P$ is obtained from
the vertex triangulation by bending the 4-set. Then there is a looped simple
graph $G^{\prime}$ obtained from $G$ by some sequence of local
complementations and loop complementations, such that the resulting compatible
automorphism $\beta:M[IAS(G)]\rightarrow M[IAS(G^{\prime})]$ has the property
that $u,v,w,x$ is a bent 4-path in $\beta(P)$. We have just verified that
there is a matroid automorphism $\alpha:M[IAS(G^{\prime})]\rightarrow
M[IAS(G^{\prime})]$ under which the image of $\beta(P)$ is the vertex
triangulation. Then $\beta^{-1}\alpha\beta$ is an automorphism of $M[IAS(G)]$
that satisfies the proposition.
\end{proof}

\subsection{Theorem \ref{noncom}}

Most of our proof of Theorem \ref{noncom} is devoted to showing that every
non-vertex triangulation of an isotropic matroid can be built by interchanging
parallels and bending 4-sets.

\begin{lemma}
\label{lone}Let $G$ be a looped simple graph, and let $P$ be a non-vertex
triangulation of $W(G)$. Suppose no non-vertex triple of $P$ contains two
elements of $W(G)$ that correspond to the same vertex of $G$. Then there is a
sequence $\Sigma$ of local complementations and loop complementations such
that the graph $G^{\prime}$ obtained by applying $\Sigma$ to $G$ has an
unlooped degree-2 vertex $w$ with $\beta_{\Sigma}^{-1}(\{w_{\chi},v_{\phi
},x_{\phi}\})\in P$. Here $N_{G^{\prime}}(w)=\{v,x\}$ and $\beta_{\Sigma
}:M[IAS(G)]\rightarrow M[IAS(G^{\prime})]$ is the compatible isomorphism
induced by $\Sigma$.
\end{lemma}

\begin{proof}
To reduce the number of cases that must be considered, we perform loop
complementations to remove all loops in $G$.

Let $v$ be a vertex of $G$ such that $\{v_{\phi},v_{\chi},v_{\psi}\}\notin P$.
Then $P$ contains a triple $\{v_{\phi},a_{\gamma},b_{\delta}\}$ with $a\neq
b\neq v\neq a$ and $\gamma,\delta\in\{\phi,\chi,\psi\}$. This triple is either
a circuit of $M[IAS(G)]$ or the union of two disjoint circuits, so the
corresponding columns of $IAS(G)$ must sum to 0.

If $\gamma=\phi$ then the $b_{\delta}$ column of $IAS(G)$ must have nonzero
entries in the $a$ and $v$ columns, and not in any other columns; necessarily
then $\delta=\chi$ and $N(b)=\{a,v\}$. Similarly, if $\delta=\phi$ then
$\gamma=\chi$ and $N(a)=\{b,v\}$.

Suppose now that $\gamma=\psi$. The $a$ entry of the $v_{\phi}$ column of
$IAS(G)$ is 0, and the $a$ entry of the $a_{\psi}$ column is 1, so the $a$
entry of the $b_{\delta}$ column must be 1. It follows that $a$ and $b$ are
neighbors in $G$, so the $b$ entry of the $a_{\psi}$ column of $IAS(G)$ is 1.
Then the $b$ entry of the $b_{\delta}$ column must also be 1, so $\delta=\psi
$. The $v$ entry of the $v_{\phi}$ column of $IAS(G)$ is 1, so precisely one
of $a,b$ is a neighbor of $v$; say $a\in N(v)$ and $b\notin N(v)$. All in all,
we have $\gamma=\delta=\psi$, $v\in N(a)$ and $N(b)=(N(a)\cup\{a\})-\{b,v\}$.
It follows that in $G_{s}^{b}$, $a$ is an unlooped degree-2 vertex whose only
neighbors are $b$ and $v$. Theorems \ref{thmloop} and \ref{lcinv} tell us that
there is a compatible isomorphism $\beta:M[IAS(G)]\rightarrow M[IAS(G_{s}%
^{b})]$ whose associated map $f_{\beta}:V(G)\rightarrow S_{3}$ has $f_{\beta
}(v)=1$, $f_{\beta}(a)=(\chi\psi)$ and $f_{\beta}(b)=(\phi\psi)$, so
$\beta(\{v_{\phi},a_{\gamma},b_{\delta}\})=\{v_{\phi},a_{\chi},b_{\phi}\}$.

Finally, suppose $\gamma=\chi$. The $a$ entry of the $v_{\phi}$ column of
$IAS(G)$ is 0, and the $a$ entry of the $a_{\chi}$ column is 0, so the $a$
entry of the $b_{\delta}$ column must be 0. It follows that $a$ and $b$ are
not neighbors in $G$, so the $b$ entry of the $a_{\chi}$ column of $IAS(G)$ is
0. Then the $b$ entry of the $b_{\delta}$ column must also be 0, so
$\delta=\chi$. The $v$ entry of the $v_{\phi}$ column of $IAS(G)$ is 1, so
precisely one of $a,b$ is a neighbor of $v$; say $a\in N(v)$ and $b\notin
N(v)$. All in all, we have $\gamma=\delta=\chi$, $v\in N(a)$ and
$N(b)=N(a)-\{v\}$. If $N(b)$ is empty then the only columns of $IAS(G)$ with
nonzero entries in the $b$ row are the $b_{\phi}$ and $b_{\psi}$ columns, so
$b_{\phi}$ and $b_{\psi}$ must appear together in a\ triple of $P$; this must
be a non-vertex triple as it does not contain $b_{\chi}$, so it violates the
hypothesis that no non-vertex triple of $P$ contains two elements of $W(G)$
corresponding to the same vertex of $G$. By contradiction, then, $N(b)$ is not
empty. If $y\in N(b)$ then Theorems \ref{thmloop} and \ref{lcinv} tell us that
there is a compatible isomorphism $\beta:M[IAS(G)]\rightarrow M[IAS((G_{s}%
^{y})_{s}^{a}]$ whose associated map $f_{\beta}:V(G)\rightarrow S_{3}$ has
$f_{\beta}(v)=(\chi\psi)$, $f_{\beta}(a)=(\phi\psi)(\chi\psi)=(\phi\psi\chi)$
and $f_{\beta}(b)=(\chi\psi)^{2}=1$, so $\beta(\{v_{\phi},a_{\gamma}%
,b_{\delta}\})=\{v_{\phi},a_{\phi},b_{\chi}\}$.
\end{proof}

\begin{lemma}
\label{ltwo}Let $G$ be a looped simple graph, and let $P$ be a non-vertex
triangulation of $W(G)$. Suppose no non-vertex triple of $P$ contains two
elements of $W(G)$ that correspond to the same vertex of $G$. Then either
there is a bent 4-set in $P$, or there is a bent 4-set in a non-vertex
triangulation $P^{\prime}$ obtained from $P$ by interchanging two parallel
elements of $M[IAS(G)]$.
\end{lemma}

\begin{proof}
By Lemma \ref{lone}, after local complementations and loop complementations we
may presume that $G$ has no looped vertex, $G$ has a degree-2 vertex $w$, and
that $P$ includes the non-vertex triple $\{w_{\chi},v_{\phi},x_{\phi}\}$ where
$N(w)=\{v,x\}$. Then $P$ also includes a triple $\{w_{\phi},y_{\gamma
},z_{\delta}\}$ with $w\neq y\neq z\neq w$. As the $w$ entry of the $w_{\phi}$
column of $IAS(G)$ is 1, either the $y_{\gamma}$ or the $z_{\delta}$ column
also has its $w$ entry equal to 1; say it is the $z_{\delta}$ column. Then
$\delta\in\{\chi,\psi\}$ and $z\in N(w)$, so $z\in\{v,x\}$; say $z=v$. Notice
that if $y=x$ then $\gamma\neq\phi$, as $x_{\phi}$ appears in the triple
$\{w_{\chi},v_{\phi},x_{\phi}\}$; but then every element of $\{w_{\phi
},y_{\gamma},z_{\delta}\}$ corresponds to a column of $IAS(G)$ whose $w$ entry
is 1, an impossibility as $\{w_{\phi},y_{\gamma},z_{\delta}\}$ is a circuit or
a disjoint union of circuits in $M[IAS(G)]$. Consequently $y\neq x$.

Summing up: $P$ contains the triples $\{w_{\chi},v_{\phi},x_{\phi}\}$ and
$\{w_{\phi},y_{\gamma},v_{\delta}\}$ with $N(w)=\{v,x\}$ and $y\notin
\{v,w,x\}$.

Case 1: If $\gamma=\phi$ then since $\{w_{\phi},y_{\phi},v_{\delta}\}$ is a
triple of $P$, it must be that $\delta=\chi$ and $N(v)=\{w,y\}$. Among the
elements of $W(G)$ not included in $\{w_{\phi},y_{\phi},v_{\chi}\}$ or
$\{w_{\chi},v_{\phi},x_{\phi}\}$, only four correspond to columns of $IAS(G)$
with nonzero $w$ entries, namely $v_{\psi}$, $w_{\psi}$, $x_{\chi}$ and
$x_{\psi}$. Consequently two of these must appear together in one triple of
$P$, and the other two in another triple. Similarly, among the elements of
$W(G)$ not included in $\{w_{\phi},y_{\phi},v_{\chi}\}$ or $\{w_{\chi}%
,v_{\phi},x_{\phi}\}$, only four correspond to columns of $IAS(G)$ with
nonzero $v$ entries, namely $v_{\psi}$, $w_{\psi}$, $y_{\chi}$ and $y_{\psi}$;
these also must appear in two triples of $P$, with two in each triple. By
hypothesis, $x_{\chi}$ and $x_{\psi}$ do not appear in the same triple of $P$;
nor do $y_{\chi}$ and $y_{\psi}$. Consequently $P$ has two triples of the form%
\[
\{\text{one of }x_{\chi},x_{\psi}\}\cup\{\text{one of }y_{\chi},y_{\psi}%
\}\cup\{\text{one of }v_{\psi},w_{\psi}\}\text{.}%
\]
As $N(v)\cup N(w)=\{v,w,x,y\}$ and the sum of the columns of $IAS(G)$
corresponding to a triple of $P$ must be 0, it follows that
$N(x)-\{v,w,x,y\}=N(y)-\{v,w,x,y\}$.

If $x$ and $y$ are not adjacent in $G$ then among the elements $x_{\chi}$,
$x_{\psi}$, $y_{\chi}$, $y_{\psi}$, $v_{\psi}$ and $w_{\psi}$, the only ones
that correspond to columns of $IAS(G)$ with nonzero $x$ entries are $x_{\psi}$
and $w_{\psi}$; so they must appear in the same triple. Similarly, $y_{\psi}$
and $v_{\psi}$ must appear in the same triple. Consequently $\{x_{\psi
},w_{\psi},y_{\chi}\}$ and $\{x_{\chi},v_{\psi},y_{\psi}\}$ are triples of
$P$, so $y,v,w,x$ is a bent 4-path in $P$.

On the other hand, if $x$ and $y$ are adjacent in $G$ then among the elements
$x_{\chi}$, $x_{\psi}$, $y_{\chi}$, $y_{\psi}$, $v_{\psi}$ and $w_{\psi}$, the
only ones that correspond to columns of $IAS(G)$ with $x$ entries equal to 0
are $x_{\chi}$\ and $v_{\psi}$; these cannot appear in the same triple of $P$.
Similarly, the only ones that correspond to columns of $IAS(G)$ with $y$
entries equal to 0 are $y_{\chi}$\ and $w_{\psi}$; and these cannot appear in
the same triple. Consequently $\{x_{\chi},w_{\psi},y_{\psi}\}$ and $\{x_{\psi
},y_{\chi},v_{\psi}\}$ are triples of $P$. In this situation $y,v,w$ and $x$
are the vertices of a 4-cycle of $G$, in this order, with $v$ and $w$ of
degree two and $N(x)-\{w,y\}=N(y)-\{v,x\}$. Then $y$, $w$, $v$, $x$ is\ a
matched 4-path in $G^{vw}$. Corollary \ref{pivinv} tells us that there is a
compatible isomorphism $\beta:M[IAS(G)]\rightarrow M[IAS(G^{vw})]$ whose
associated map $f_{\beta}:V(G)\rightarrow S_{3}$ has $f_{\beta}(v)=f_{\beta
}(w)=(\phi\chi)$ and $f_{\beta}(x)=f_{\beta}(y)=1$, so $\beta(P)$ is a
triangulation of $W(G^{vw})$ with triples $\{w_{\phi},v_{\chi},x_{\phi}\}$,
$\{w_{\chi},y_{\phi},v_{\phi}\}$, $\{x_{\chi},w_{\psi},y_{\psi}\}$ and
$\{x_{\psi},y_{\chi},v_{\psi}\}$. Hence $y$, $w$, $v$, $x$ is a bent 4-path in
$\beta(P)$.

Case 2: If $\gamma=\psi$ then since $\{w_{\phi},y_{\psi},v_{\delta}\}$ is a
triple of $P$ and the $y$ entry of the column of $IAS(G)$ corresponding to
$w_{\phi}$ is 0, it must be that $v\in N(y)$. Then the $v$ entry of the column
corresponding to $y_{\gamma}$ is 1, so $\delta=\psi$ and $N(y)=(N(v)\cup
\{v\})-\{y,w\}$. Theorems \ref{thmloop} and \ref{lcinv} tell us that there is
a compatible isomorphism $\beta:M[IAS(G)]\rightarrow M[IAS(G_{s}^{y})]$ whose
associated map $f_{\beta}:V(G)\rightarrow S_{3}$ has $f_{\beta}(y)=(\phi\psi
)$, $f_{\beta}(z)=(\chi\psi)$ for $z\in N(y)$, and $f_{\beta}(z)=1$ for
$z\notin N(y)\cup\{y\}$. As $w\notin N(y)\cup\{y\}$, it follows that
$\beta(P)$ is a triangulation of $W(G_{s}^{y})$ that contains the triples
$\{w_{\chi},v_{\phi},x_{\phi}\}$ and $\{w_{\phi},y_{\phi},v_{\chi}\}$. That
is, Case 1 holds in $G_{s}^{y}$.

Case 3: Suppose $\gamma=\chi$. As $\{w_{\phi},y_{\gamma},v_{\delta}\}$ is a
triple of $P$ and the $y$ entry of the column of $IAS(G)$ corresponding to
$w_{\phi}$ is 0, it must be that $v\not \in N(y)$. Then the $v$ entry of the
column corresponding to $y_{\gamma}$ is 0, so $\delta=\chi$ and
$N(y)=N(v)-\{w\}$. If $N(y)$ is empty then the $y_{\phi}$ and $y_{\psi}$
columns of $IAS(G)$ are the only ones with nonzero $y$ entries, so $y_{\phi}$
and $y_{\psi}$ must appear in the same triple of $P$. This triple doesn't
contain $y_{\chi}$, so it is not a vertex triple; but no such triple exists,
by hypothesis. Consequently $N(y)$ is not empty.

If $x\neq u\in N(y)$ then $u\notin\{x,v\}=N(w)$, so Theorems \ref{thmloop} and
\ref{lcinv} tell us that there is a compatible isomorphism $\beta
:M[IAS(G)]\rightarrow M[IAS(G_{s}^{u})]$ whose associated map $f_{\beta
}:V(G)\rightarrow S_{3}$ has $f_{\beta}(u)=(\phi\psi)$, $f_{\beta}(w)=1$,
$f_{\beta}(v)=$ $f_{\beta}(y)=(\chi\psi)$ and $f_{\beta}(z)\in\{1,(\chi
\psi)\}$ for $z\notin\{u,v,w,y\}$. As $P$ contains the triples $\{w_{\chi
},v_{\phi},x_{\phi}\}$ and $\{w_{\phi},y_{\chi},v_{\chi}\}$, it follows that
$\beta(P)$ contains the triples $\{w_{\chi},v_{\phi},x_{\phi}\}$ and
$\{w_{\phi},y_{\psi},v_{\psi}\}$. That is, Case 2 holds in $G_{s}^{u}$.

It remains to consider the possibility that $N(y)=\{x\}$. Then the only
columns of $IAS(G)$ with nonzero $y$ entries are those corresponding to
$y_{\phi}$, $y_{\psi}$, $x_{\chi}$ and $x_{\psi}$, so there must be two
triples of $P$ each of which contains one of $y_{\phi}$, $y_{\psi}$ and one of
$x_{\chi}$, $x_{\psi}$. Also, the fact that $\{w_{\phi},y_{\chi},v_{\chi}\}$
is a triple of $P$ implies that $N(v)=\{w,x\}$; hence the only columns of
$IAS(G)$ with nonzero $v$ entries are those corresponding to $v_{\phi}$,
$v_{\psi}$, $w_{\chi}$, $w_{\psi}$, $x_{\chi}$ and $x_{\psi}$. As $\{w_{\chi
},v_{\phi},x_{\phi}\}$ is a triple of $P$ there must be two triples of $P$
each of which contains one of $v_{\psi}$, $w_{\psi}$ and one of $x_{\chi}$,
$x_{\psi}$. Consequently $P$ has two triples of the form
\[
\{\text{one of }x_{\chi},x_{\psi}\}\cup\{\text{one of }y_{\phi},y_{\psi}%
\}\cup\{\text{one of }v_{\psi},w_{\psi}\}\text{.}%
\]
The columns of $IAS(G)$ corresponding to $v_{\psi}$ and $w_{\psi}$\ both have
nonzero $x$ entries, so $x_{\psi}$ and $y_{\psi}$ cannot appear in the same
triple. Consequently these two triples are
\[
\{x_{\chi},y_{\psi}\}\cup\{\text{one of }v_{\psi},w_{\psi}\}\text{ and
}\{x_{\psi},y_{\phi}\}\cup\{\text{one of }v_{\psi},w_{\psi}\}\text{.}%
\]
It follows that $N(x)=\{v,w,y\}$ and the subgraph of $G$ induced by
$\{v,w,x,y\}$ is an entire connected component of $G$. See Figure
\ref{lcmatfi2}.

Notice that $N(v)=\{w,x\}$ and $N(w)=\{v,x\}$, so $v$ and $w$ are adjacent
twins, and $v_{\psi}$ and $w_{\psi}$ are parallel in $M[IAS(G)]$.
Interchanging $v_{\psi}$ and $w_{\psi}$ if necessary, we may presume that
$\{x_{\chi},y_{\psi},v_{\psi}\}$ and $\{x_{\psi},y_{\phi},w_{\psi}\}$ are both
triples of $P$. Theorems \ref{thmloop} and \ref{lcinv} tell us that there is a
compatible isomorphism $\beta:M[IAS(G)]\rightarrow M[IAS(G_{s}^{w})]$ whose
associated map $f_{\beta}:V(G)\rightarrow S_{3}$ has $f_{\beta}(w)=(\phi\psi
)$, $f_{\beta}(v)=f_{\beta}(x)=(\chi\psi)$ and $f_{\beta}(z)=1$ for
$z\notin\{v,w,x\}$. Consequently $\beta(P)$ contains $\beta(\{w_{\chi}%
,v_{\phi},x_{\phi}\})=\{w_{\chi},v_{\phi},x_{\phi}\}$, $\beta(\{w_{\phi
},y_{\chi},v_{\chi}\})=\{w_{\psi},y_{\chi},v_{\psi}\}$, $\beta(\{x_{\chi
},y_{\psi},v_{\psi}\})=\{x_{\psi},y_{\psi},v_{\chi}\}$ and $\beta(\{x_{\psi
},y_{\phi},w_{\psi}\})=\{x_{\chi},y_{\phi},w_{\phi}\}$. It follows that
$v,w,x,y$ is a bent 4-path in $\beta(P)$.
\end{proof}

%

\begin{figure}
[ptb]
\begin{center}
\includegraphics[
trim=1.331449in 8.566759in 4.019836in 1.610119in,
height=0.6469in,
width=2.3869in
]%
{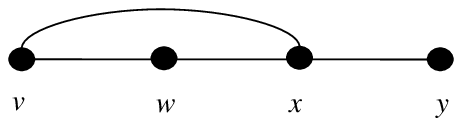}%
\caption{ The situation considered at the end of the proof of Lemma
\ref{ltwo}.}%
\label{lcmatfi2}%
\end{center}
\end{figure}

\begin{definition}
If $G$ is a looped simple graph and $P$ is a triangulation of $W(G)$ then the
\emph{index} of $P$ is $\left\Vert P\right\Vert =\left\vert \{\text{non-vertex
triples of }P\}\right\vert $.
\end{definition}

\begin{proposition}
\label{constr}Let $P$ be a non-vertex triangulation of $W(G)$. Then there are
an integer $k\in\{1,...,\left\Vert P\right\Vert \}$, a sequence $G=H_{0}%
,...,H_{k}$ of graphs and a sequence $P=P_{0},...,P_{k}$ of triangulations
such that:

\begin{enumerate}
\item If $1\leq i\leq k$ then $H_{i}$ is obtained from $H_{i-1}$ through some
(possibly empty) sequence of local complementations and loop complementations.

\item If $1\leq i<k$ then $P_{i}$ is a non-vertex triangulation of $W(H_{i}).$

\item $P_{k}$ is the vertex triangulation of $W(H_{k})$.

\item If $1\leq i\leq k$ then $\left\Vert P_{i}\right\Vert \in\{\left\Vert
P_{i-1}\right\Vert ,\left\Vert P_{i-1}\right\Vert -1,\left\Vert P_{i-1}%
\right\Vert -2,\left\Vert P_{i-1}\right\Vert -4\}$.

\item If $\left\Vert P_{i}\right\Vert \in\{\left\Vert P_{i-1}\right\Vert
,\left\Vert P_{i-1}\right\Vert -1,\left\Vert P_{i-1}\right\Vert -2\}$\ then
$P_{i}$ is obtained from $P_{i-1}$ by interchanging two parallel elements of
$M[IAS(H_{i-1})]$.

\item If $\left\Vert P_{i}\right\Vert =\left\Vert P_{i-1}\right\Vert $, then
$i<k$ and $\left\Vert P_{i+1}\right\Vert =\left\Vert P_{i}\right\Vert -4$.

\item If $\left\Vert P_{i}\right\Vert =\left\Vert P_{i-1}\right\Vert
-4$,\ then $P_{i}$ is obtained from $P_{i-1}$ by replacing the triples
corresponding to a bent 4-set with the four corresponding vertex triples.
\end{enumerate}
\end{proposition}

\begin{proof}
Suppose $v$ is a vertex of $G$ such that $\{v_{\phi},v_{\chi},v_{\psi}\}\notin
P$ and two of $v_{\phi},v_{\chi},v_{\psi}$ appear together in a single triple
of $P$. Then the third element of this triple is parallel to the third of
$v_{\phi},v_{\chi},v_{\psi}$. Interchanging these two parallels transforms
this triple into the vertex triple corresponding to $v$, and may also
transform another non-vertex triple of $P$ into a vertex triple, so the
resulting triangulation $P^{\prime}$ has $\left\Vert P^{\prime}\right\Vert
\in\{\left\Vert P\right\Vert -1,\left\Vert P\right\Vert -2\}$. If there is no
such vertex $v$, then Lemma \ref{ltwo} applies.
\end{proof}

Propositions \ref{triang} and \ref{constr} tell us that if $P$ is a non-vertex
triangulation of $M[IAS(G)]$, then there is an automorphism $\alpha_{P}$ of
$M[IAS(G)]~$such that $\alpha_{P}(P)$ is the vertex triangulation. Theorem
\ref{noncom} follows, for if $\gamma:M[IAS(G_{1})]\rightarrow M[IAS(G_{2})]$
is a non-compatible isomorphism and $P$ is the image of the vertex
triangulation of $W(G_{1})$ under $\gamma$, then $\alpha_{P}\circ
\gamma:M[IAS(G_{1})]\rightarrow M[IAS(G_{2})]$ is a compatible isomorphism.

\section{Delta-matroids and isotropic systems}

The results of this paper show that the theory of binary matroids contains
\textquotedblleft conceptual imbeddings\textquotedblright\ of the theories of
graphic delta-matroids and isotropic systems, two interesting and useful
theories studied by Bouchet in the 1980s and 1990s. Bouchet later introduced a
third theory, involving \emph{multimatroids}, to unify these two. Using
terminology of \cite{B1}, we can summarize our \textquotedblleft conceptual
imbeddings\textquotedblright\ by saying two things. First, if $G$ is a looped
simple graph then $M[IAS(G)]$ is a binary matroid that shelters the 3-matroid
associated with an isotropic system with fundamental graph $G$, and the
submatroid $M[IA(G)]$ shelters the 2-matroid associated with a delta-matroid
with fundamental graph $G$. (\textquotedblleft Sheltering\textquotedblright%
\ is a way of containing; that's why we use the term \textquotedblleft
imbedding.\textquotedblright) Second, matroidal properties of $M[IAS(G)]$
provide new explanations of the properties of graphic delta-matroids and
isotropic systems; that's what makes the imbeddings \textquotedblleft
conceptual.\textquotedblright\ For instance, compatible isomorphisms of
isotropic matroids provide a new explanation of the significance of isotropic
systems, using the fact that certain kinds of basis exchanges correspond to
local complementations. Compatible isomorphisms also provide a new way to
conceptualize the work of Brijder and Hoogeboom \cite{BH3, BH} on the
connection between $S_{3}$ and certain operations on delta-matroids.

\begin{definition}
\cite{bouchet1987} If $G$ is a looped simple graph, then the delta-matroid
associated to $G$ is
\[
D(G)=\{S\subseteq V(G)\mid\text{the submatrix }A(G)[S]\text{ is nonsingular
over }GF(2)\}.
\]

\end{definition}

Here $A(G)[S]$ denotes the principal submatrix of $A(G)$ obtained by removing
all rows and columns corresponding to vertices $v\notin S$. Observe that
\[
D(G)=\{S\subseteq V(G)\mid\{s_{\chi}\mid s\in S\}\cup\{v_{\phi}\mid v\notin
S\}\text{ is a basis of }M[IAS(G)]\}\text{,}%
\]
so the matroid $M[IAS(G)]$ determines $D(G)$. (The index $\psi$ does not
appear in this description of $D(G)$, so the submatroid $M[IA(G)]$ actually
contains enough information to determine $D(G)$.) Moreover, if $G_{1}$ and
$G_{2}$ are looped simple graphs and there is a compatible isomorphism
$\beta:M[IAS(G_{1})]\rightarrow M[IAS(G_{2})]$ with $f_{\beta}(v)(\psi)=\psi$
$\forall v\in V(G_{1})$, then the set $X=\{v\in V(G_{1})\mid f_{\beta}%
(v)\neq1\}$ has the property that%
\[
D(G_{2})=\{S\Delta X\mid S\in D(G_{1})\}\text{.}%
\]
Consequently, the significance of symmetric difference (also called
\textquotedblleft twisting\textquotedblright) for the theory of graphic
delta-matroids follows from the results of Section 4, regarding compatible
isomorphisms of $M[IAS(G)]$ and $M[IA(G)]$.

It takes only a little more work to see how $M[IAS(G)]$ determines an
isotropic system.

\begin{definition}
If $G$ is a looped simple graph then the \emph{sub-transversals} of $W(G)$ are
the elements of $\mathcal{S}(W(G))=\{S\subseteq W(G)\mid\left\vert
S\cap\{v_{\phi},v_{\chi},v_{\psi}\}\right\vert \leq1$ $\forall v\in V(G)\}.$
\end{definition}

Recall that the power set $\mathcal{P}(W(G))$ is a vector space over $GF(2)$,
with symmetric difference used for addition. Let $Q$ be the subspace of
$\mathcal{P}(W(G))$ spanned by vertex triples. Then each element of the
quotient space $\mathcal{P}(W(G))/Q$ includes one element of $\mathcal{S}%
(W(G))$, so we may identify $\mathcal{S}(W(G))$ with $\mathcal{P}(W(G))/Q$.
Denote the resulting addition in $\mathcal{S}(W(G))$ by $\boxplus$.

Recall also that the \emph{cycle space} $Z(M[IAS(G)])$ is the $GF(2)$-subspace
of $\mathcal{P}(W(G))$ consisting of the subsets of $W(G)$ that correspond to
sets of columns of $IAS(G)$ whose sum is $0$.

\begin{definition}
\label{subt}A \emph{transverse cycle} of $G$ is an element of%
\[
\mathcal{L}(G)=\mathcal{S}(W(G))\cap Z(M[IAS(G)]).
\]

\end{definition}

If $X\subseteq V(G)$ and $S\in\mathcal{S}(W(G))$ then $X\cdot S$ denotes
$S\cap\{v_{\phi},v_{\chi},v_{\psi}\mid v\in X\}$.

\begin{proposition}
\label{cychar}Let $\Phi(G)=\{v_{\phi}\mid v\in V(G)\}$, and let $\Psi
(G)=\{v_{\psi}\mid v\in V(G)$ is looped $\}\cup\{v_{\chi}\mid v\in V(G)$ is
unlooped $\}$. Then
\[
\mathcal{L}(G)=\{(X\cdot\Psi(G))\boxplus(N(X)\cdot\Phi(G))\mid X\subseteq
V(G)\}\text{.}%
\]

\end{proposition}

\begin{proof}
Let $S\in\mathcal{S}(W(G))$. Then $S\in\mathcal{L}(G)$ if and only if for
every $v\in V(G)$, the sum of the $v$ entries of the columns of $IAS(G)$
corresponding to elements of $S$ is $0$. That is, if we let $S_{\phi}=\{v\in
V(G)\mid v_{\phi}\in S\}$, $S_{\chi}^{\ell}=\{$looped $v\in V(G)\mid v_{\chi
}\in S\}$, $S_{\psi}^{\ell}=\{$looped $v\in V(G)\mid v_{\psi}\in S\}$,
$S_{\chi}=\{$unlooped $v\in V(G)\mid v_{\chi}\in S\}$, and $S_{\psi}%
=\{$unlooped $v\in V(G)\mid v_{\psi}\in S\}$ then $S\in\mathcal{L}(G)$ if and
only if the following conditions are met:

\begin{enumerate}
\item For every $v\in S_{\chi}\cup S_{\psi}^{\ell}$, $\left\vert
N(v)\cap(S-S_{\phi})\right\vert $ is even.

\item For every $v\in S_{\phi}\cup S_{\chi}^{\ell}\cup S_{\psi}$, $\left\vert
N(v)\cap(S-S_{\phi})\right\vert $ is odd.

\item For every $v\in V(G)$ with $\{v_{\phi},v_{\chi},v_{\psi}\}\cap
S=\varnothing$, $\left\vert N(v)\cap(S-S_{\phi})\right\vert $ is even.
\end{enumerate}

It follows that $S\in\mathcal{L}(G)$ if and only if $S=(X\cdot\Psi
(G))\boxplus(N(X)\cdot\Phi(G))$, with $X=S-S_{\phi}$.
\end{proof}

As $\Phi(G)$ and $\Psi(G)$ are disjoint elements of $\mathcal{S}(W(G))$, and
each is of size $\left\vert V(G)\right\vert $, they satisfy Bouchet's
definition of \emph{supplementary vectors} \cite{Bi2}. It follows from
Proposition \ref{cychar} that $\mathcal{L}(G)$ is an isotropic system with
fundamental graph $G$. The basic theorem of isotropic systems -- that two
simple graphs are locally equivalent if and only if they are fundamental
graphs of strongly isomorphic isotropic systems -- now follows immediately
from Theorem \ref{lcintro4}.

It is worth taking a\ moment to observe that even though $\mathcal{L}(G)$
includes only the transverse cycles of $G$, it contains enough information to
determine $G$, and hence also $M[IAS(G)]$. The reason is simple: For each
$v\in V(G)$, $\mathcal{L}(G)$ contains precisely one transverse cycle
$\zeta_{v}\subset\{v_{\chi},v_{\psi}\}\cup\{w_{\phi}\mid v\neq w\in V(G)\}$.
The open neighborhood of $v$ is $N(v)=\{w\mid w_{\phi}\in\zeta_{v}\}$, and $v$
is looped if and only if $v_{\psi}\in\zeta_{v}$.

Before proceeding, we take another moment to expand on the following comment
of Bouchet \cite{B1}:

\begin{quotation}
The theory of isotropic systems can be considered as an extension of the
theory of binary matroids, whereas delta-matroids extend arbitrary matroids.
However delta-matroids do not generalize isotropic systems.
\end{quotation}

Jaeger showed that every binary matroid can be represented by some symmetric
$GF(2)$-matrix, or equivalently, by the adjacency matrix of some looped simple
graph \cite{J1}. (This result is also discussed in \cite{BHT}.) It follows
that every binary matroid can be extended to some isotropic matroid. As the
theory of isotropic systems is equivalent to the theory of isotropic matroids,
this confirms the first part of Bouchet's comment. On the other hand, all
isotropic matroids are binary so the theory of isotropic systems can also be
considered to be a \emph{subset} of the theory of binary matroids, rather than
an extension.

The second sentence of Bouchet's comment seems questionable. If $G$ is a
looped simple graph then $G$ is completely determined by $D(G)$: a vertex $v$
is looped if and only $\{v\}\in D(G)$, two looped vertices $v$ and $w$ are
adjacent if and only if $\{v,w\}\notin D(G)$, and otherwise two vertices $v$
and $w$ are adjacent if and only if $\{v,w\}\in D(G)$. Consequently, $D(G)$
also determines the isotropic systems with fundamental graph $G$, up to strong
isomorphism. All isotropic systems have fundamental graphs, and there are
non-graphic delta-matroids, so it would certainly seem that in some sense,
delta-matroids \emph{do} generalize isotropic systems. The reader interested
in a detailed discussion of this point will appreciate\ a recent paper of
Brijder and Hoogeboom \cite{BH}.

\section{Some properties of isotropic matroids}

In this section we quickly survey several basic properties of isotropic
matroids. One basic property was noted above: every binary matroid is a
submatroid of some isotropic matroid. A simpler basic property is that the
isotropic matroid of a one-vertex graph is not connected: it consists of a
loop and a pair of parallel non-loops. For larger graphs, though, we have the following.

\begin{proposition}
\label{connected}Let $G$ be a looped simple graph with two or more vertices.
Then $M[IAS(G)]$ is a connected matroid if and only if $G$ is a connected graph.
\end{proposition}

\begin{proof}
Suppose first that $G$ is connected. If $v\in V(G)$ then the columns of
$IAS(G)$ corresponding to $v_{\phi}$, $v_{\chi}$, $v_{\psi}$ are nonzero, and
sum to $0$; hence $\{v_{\phi}$, $v_{\chi}$, $v_{\psi}\}$ is a circuit of
$M[IAS(G)]$. Let $\Phi$ denote the basis $\{w_{\phi}\mid w\in V(G)\}$ of
$M[IAS(G)]$. If $v$ neighbors $w$ in $G$ then $w_{\phi}$ and $v_{\chi}$ are
both elements of the fundamental circuit $C(v_{\chi},\Phi)$, so $\{v_{\phi}$,
$v_{\chi}$, $v_{\psi}\}$ and $\{w_{\phi}$, $w_{\chi}$, $w_{\psi}\}$ are
contained in the same component of $M[IAS(G)]$. As this holds for all
neighbors and $G$ is connected, we conclude that all elements of $M[IAS(G)]$
lie in the same component.

Now suppose $G$ is not connected, and let $K$ be a connected component of $G$.
If $C$ is a set of columns of $IAS(G)$ whose sum is $0$, then the subset
$C_{K}=\{x\in C\mid x$ corresponds to a vertex of $K\}$ must sum to $0$, as no
element of $C_{K}$ has a nonzero entry in any row corresponding to a vertex
outside of $K$. It follows that every circuit of $M[IAS(G)]$ that intersects
$M[IAS(K)]$ is contained in $M[IAS(K)]$, so $M[IAS(G)]$ is not connected: it
is the direct sum of $M[IAS(K)]$ and $M[IAS(G-K)]$.
\end{proof}

Note that the argument of the second paragraph of the proof of Proposition
\ref{connected} implies that if $G$ is a looped simple graph with connected
components $K_{1},...,K_{c}$ then
\[
M[IAS(G)]=%
{\displaystyle\bigoplus\limits_{i=1}^{c}}
M[IAS(K_{i})]\text{.}%
\]

\subsection{Minors}

Given the discussion of Section 6, it is no surprise that some properties of
isotropic matroids are suggested by properties of delta-matroids and isotropic
systems. For instance, local complementation and vertex deletion are connected
to matroid minor operations in much the same way as they are connected to the
minor operations of isotropic systems \cite[Section 8]{Bi1}. Establishing
these connections is somewhat easier here, though, because the arguments
require only elementary linear algebra.

\begin{proposition}
\label{minor}If $A\subseteq V(G)$ then%
\[
M[IAS(G-A)]=(M[IAS(G)]/\{a_{\phi}\mid a\in A\})-\{a_{\chi},a_{\psi}\mid a\in
A\}\text{.}%
\]

\end{proposition}

\begin{proof}
If $A$ has only one element, $a$, then the lone nonzero entry of the $a_{\phi
}$ column of $M[IAS(G)]$ is a 1 in the $a$ row. The definition of matroid
contraction implies that $M[IAS(G)]/a_{\phi}$ is the matroid represented by
the submatrix of $IAS(G)$ obtained by removing the $a$ row and the $a_{\phi}$
column. Removing the $a_{\chi}$ and $a_{\psi}$ columns then yields $IAS(G-a)$,
so the proposition holds when $A=\{a\}$. The general case is verified by
removing the elements of $A$ one at a time.
\end{proof}

\begin{corollary}
$G$ can be reconstructed from $M[IAS(G)]$.
\end{corollary}

\begin{proof}
A vertex $v$ is looped in $G$ if and only if $v_{\psi}$ is a loop in
$M[IAS(G-(V(G)-\{v\}))]$, and two vertices $v$ and $w$ are adjacent in $G$ if
and only if $M[IAS(G-(V(G)-\{v,w\}))]$ is a connected matroid.
\end{proof}

\begin{corollary}
$M[IAS(G)]$ is a regular matroid if and only if $G$ has no connected component
with more than two vertices.
\end{corollary}

\begin{proof}
If every connected component of $G$ has one or two vertices, then $M[IAS(G)]$
is a direct sum of submatroids of size three or six. The smallest binary
matroids that are not regular have seven elements, so $M[IAS(G)]$ is a direct
sum of regular matroids.

On the other hand, if $G$ has a connected component with three or more
vertices then there is a subset $A\subset V(G)$ such that $G-A$ is isomorphic
to a looped version of either the complete graph $K_{3}$ or the path $P_{3}$.
Then $IAS(G-A)$ is a $3\times9$ matrix with a submatrix whose columns can be
permuted to yield%
\[%
\begin{pmatrix}
1 & 0 & 0 & 1 & 1 & 0 & 1\\
0 & 1 & 0 & 1 & 0 & 1 & 1\\
0 & 0 & 1 & 0 & 1 & 1 & 1
\end{pmatrix}
.
\]
Consequently, the Fano matroid is\ a submatroid of $M[IAS(G-A)]$. As
$M[IAS(G-A)]$ is a minor of $M[IAS(G)]$, it follows that $M[IAS(G)]$ is not regular.
\end{proof}

The next two results are obtained by combining Proposition \ref{minor} with
Theorem \ref{lcinv} and Corollary \ref{pivinv}.

\begin{corollary}
\label{lcminor}Let $a$ be a vertex of a looped simple graph $G$. Then%
\[
M[IAS(G_{ns}^{a}-a)]=\left\{
\begin{array}
[c]{cc}%
(M[IAS(G)]/a_{\psi})-a_{\phi}-a_{\chi}\text{,} & \text{if }a\text{ is\ not
looped in }G\\
& \\
(M[IAS(G)]/a_{\chi})-a_{\phi}-a_{\psi}\text{,} & \text{if }a\text{ is looped
in }G
\end{array}
\right.  .
\]

\end{corollary}

\begin{corollary}
\label{pivminor}Let $a$ be a vertex of a looped simple graph $G$, and let $b$
be a neighbor of $a$. Then%
\[
M[IAS(G^{ab}-a)]\cong\left\{
\begin{array}
[c]{cc}%
(M[IAS(G)]/a_{\chi})-a_{\phi}-a_{\psi}\text{,} & \text{if }a\text{ is\ not
looped in }G\\
& \\
(M[IAS(G)]/a_{\psi})-a_{\phi}-a_{\chi}\text{,} & \text{if }a\text{ is looped
in }G
\end{array}
\right.  .
\]

\end{corollary}

Note that $=$ appears in Corollary \ref{lcminor} because the compatible
isomorphism $\beta:M[IAS(G)]\rightarrow M[IAS(G_{ns}^{a})]$ of Theorem
\ref{lcinv} has $f_{\beta}(x)=1$ $\forall x\neq a$. In Corollary
\ref{pivminor} we write $\cong$ instead because the compatible isomorphism
$\beta:M[IAS(G)]\rightarrow M[IAS(G^{ab})]$ of Corollary \ref{pivinv} has
$f_{\beta}(b)\neq1$.

\begin{corollary}
Let $G$ be a looped simple graph, and let $M$ be a binary matroid. Then these
two statements are equivalent.

1. $M$ is isomorphic to the isotropic matroid of a graph obtained from $G$
through some sequence of local complementations, loop complementations and
vertex deletions.

2. $M$ is isomorphic to a minor of $M[IAS(G)]$ obtained by removing some
vertex triples, each vertex triple removed by contracting one element and
deleting the other two.
\end{corollary}

\begin{proof}
Recall that if $a$ is an isolated vertex of $G$ then the corresponding vertex
triple $\{a_{1},a_{2},a_{3}\}$ contains two components of $M[IAS(G)]$, a
singleton component containing a loop and a two-element component containing a
pair of parallel non-loops. The result of removing these three elements by
deletion and contraction is the same no matter which elements are deleted and
which are contracted. According to Proposition \ref{minor}, then,%
\[
(M[IAS(G)]/a_{1})-a_{2}-a_{3}=M[IAS(G-a)]
\]
no matter how the elements of the vertex triple are ordered.

Using the preceding observation for isolated vertices and Proposition
\ref{minor}, Corollary \ref{lcminor} and Corollary \ref{pivminor} for
non-isolated vertices, we deduce the equivalence asserted in the statement
from Theorem \ref{lcintro}.
\end{proof}

\subsection{The triangle property and strong maps}

Recall Definition \ref{subt}: a subtransversal of $W(G)$ is a subset that
contains no more than one element from each triple of the vertex
triangulation. The ranks of subtransversals in $M[IAS(G)]$ are connected to
each other through the \emph{triangle property}, which is part of Bouchet's
theory of isotropic systems \cite[Section 9]{Bi1}.

\begin{proposition}
\label{triangle}Suppose $r$ is the rank function of $M[IAS(G)]$, $S$ is a
subtransversal of $W(G)$ with $\left\vert S\right\vert =\left\vert
V(G)\right\vert -1$, and $v$ is the vertex of $G$ with $v_{\phi}$, $v_{\chi}$,
$v_{\psi}\notin S$. Let $S_{\phi}=S\cup\{v_{\phi}\}$, $S_{\chi}=S\cup
\{v_{\chi}\}$ and $S_{\psi}=S\cup\{v_{\psi}\}$. Then one of $S_{\phi}$,
$S_{\chi}$, $S_{\psi}$ has rank $r(S)$ in $M[IAS(G)]$, and the other two have
rank $r(S)+1$.
\end{proposition}

\begin{proof}
Complementing the loop status of $v$ has the effect of interchanging $v_{\chi
}$ and $v_{\psi}$, and this interchange does not affect the statement of the
proposition, so we may suppose without loss of generality that $v$ is looped.
Order the other vertices of $G$ as $v_{1},...,v_{n-1}$ in such a way that for
some $p\in\{1,...,n\}$, $v_{i\phi}\in S$ if and only if $i<p$. Then there is a
symmetric $(n-1-p)\times(n-1-p)$ matrix $B$ such that
\[
r(S_{\phi})=r%
\begin{pmatrix}
I & \ast & 0\\
0 & B & 0\\
0 & \rho & 1
\end{pmatrix}
\text{, }r(S_{\chi})=r%
\begin{pmatrix}
I & \ast & \ast\\
0 & B & \kappa\\
0 & \rho & 1
\end{pmatrix}
\text{ and }r(S_{\psi})=r%
\begin{pmatrix}
I & \ast & \ast\\
0 & B & \kappa\\
0 & \rho & 0
\end{pmatrix}
\text{.}%
\]
Here $r$ denotes the rank function of $M[IAS(G)]$ and also matrix rank over
$GF(2)$; $I$ is the $(p-1)\times(p-1)$ identity matrix; $\rho$ is the row
vector whose nonzero entries occur in columns such that $p\leq i\leq n-1$ and
$v_{i}$ neighbors $v$; $\kappa$ is the transpose of $\rho$; and $\ast$
indicates submatrices that do not contribute to the rank. Using elementary
column operations, we deduce that
\[
r(S_{\phi})=p+r(B)\text{, }r(S_{\chi})=p-1+r%
\begin{pmatrix}
B & \kappa\\
\rho & 1
\end{pmatrix}
\text{ and }r(S_{\psi})=p-1+r%
\begin{pmatrix}
B & \kappa\\
\rho & 0
\end{pmatrix}
\text{.}%
\]
A result mentioned by Balister, Bollob\'{a}s, Cutler, and Pebody \cite[Lemma
2]{BBCS} implies that two of the ranks $r(S_{\phi})$, $r(S_{\chi})$,
$r(S_{\psi})$ are the same, and the other is one less. As each of these ranks
is $r(S)$ or $r(S)+1$, the proposition follows.
\end{proof}

\begin{corollary}
\label{strongmap}Let $S$ and $T$ be disjoint transversals of $W(G)$, i.e.,
$S\cap T=\varnothing$ and $\left\vert S\cap\{v_{\phi},v_{\chi},v_{\psi
}\}\right\vert =1=\left\vert T\cap\{v_{\phi},v_{\chi},v_{\psi}\}\right\vert $
$\forall v\in V(G)$. For each $v\in V(G)$ let $v_{S}$ and $v_{T}$ be the
elements of $S\cap\{v_{\phi},v_{\chi},v_{\psi}\}$ and $T\cap\{v_{\phi}%
,v_{\chi},v_{\psi}\}$, respectively. Then the function $v_{S}\mapsto v_{T}$
defines a strong map from $M[IAS(G)]|S$ to $(M[IAS(G)]|T)^{\ast}$.
\end{corollary}

\begin{proof}
For $A\subseteq V(G)$ let $A_{S}=\{a_{S}\mid a\in A\}$ and $A_{T}=\{a_{T}\mid
a\in A\}$. The assertion that $v_{S}\mapsto v_{T}$ defines a strong map is
equivalent to this claim: if $v\notin A$ and the closure of $A_{S}$ in
$M[IAS(G)]|S$ includes $v_{S}$, then the closure of $A_{T}$ in
$(M[IAS(G)]|T)^{\ast}$ includes $v_{T}$.

Suppose instead that the closure of $A_{S}$ in $M[IAS(G)]|S$ includes $v_{S}$,
and the closure of $A_{T}$ in $(M[IAS(G)]|T)^{\ast}$ does not include $v_{T}$.
A fundamental property of matroid duality is that the closure of $A_{T}$ in
$(M[IAS(G)]|T)^{\ast}$ does not include $v_{T}$ if and only if the closure of
$T-A_{T}-\{v_{T}\}$ in $M[IAS(G)]|T$ does include $v_{T}$. It follows that the
closure of $U=A_{S}\cup(T-A_{T}-\{v_{T}\})$ in $M[IAS(G)]$ includes both
$v_{S}$ and $v_{T}$. $U$ is a subtransversal, though, so Proposition
\ref{triangle} tells us that its closure cannot include both $v_{S}$ and
$v_{T}$. By contradiction, then, the claim must hold.
\end{proof}

Corollary \ref{strongmap} may seem to be a merely technical result, but it
generalizes one of the most famous situations in matroid theory. If $H$ and
$K$ are dual graphs in the plane then they give rise to disjoint transversals
$S$ and $T$ of $W(G)$, where $G$ is an interlacement graph of the medial graph
shared by $H$ and $K$. In this case the strong map $v_{S}\mapsto v_{T}$ is the
familiar isomorphism between the bond matroid of $H$ and the cycle matroid of
$K$. We refer to \cite{Tra} for more details of the significance of isotropic
matroids in the theory of 4-regular graphs.

\section{Interlace polynomials and Tutte polynomials}

Motivated by problems that arise in the study of DNA sequencing, Arratia,
Bollob\'{a}s and Sorkin introduced a one-variable graph polynomial, the
\emph{vertex-nullity interlace polynomial}, in \cite{A1}. In subsequent work
\cite{A2, A} they observed that this one-variable polynomial may be obtained
from the Tutte-Martin polynomial of isotropic systems studied by\ Bouchet
\cite{Bi3, B5}, introduced an extended two-variable version of the interlace
polynomial, and observed that the interlace polynomials are given by formulas
that involve the nullities of matrices over the two-element field, $GF(2)$.
Inspired by these ideas, Aigner and van der Holst \cite{AH}, Courcelle
\cite{C} and the author \cite{Tw, T6} introduced several different variations
on the interlace polynomial theme.

All these references share the underlying presumption that although the
interlace and Tutte-Martin polynomials are connected to other graph
polynomials in some ways, they are in a general sense separate invariants. In
this section we point out that in fact, the interlace polynomials of graphs
can be derived from parametrized Tutte polynomials of isotropic matroids.

One way to define the \emph{Tutte polynomial} of $M[IAS(G)]$ is a polynomial
in the variables $s$ and $z$, given by the subset expansion%
\[
t(M[IAS(G)])=\sum_{T\subseteq W(G)}s^{r^{G}(W(G))-r^{G}(T)}z^{\left\vert
T\right\vert -r^{G}(T)}\text{.}%
\]
Here $r^{G}$ denotes the rank function of $M[IAS(G)]$. We do not give a
general account of this famous invariant of graphs and matroids here; thorough
introductions may be found in \cite{Bo, BO, D, GM}.

Tutte polynomials of graphs and matroids are remarkable both for the amount of
structural information they contain and for the range of applications in which
they appear. Some applications (electrical circuits, knot theory, network
reliability, and statistical mechanics, for instance) involve graphs or
networks whose vertices or edges have special attributes of some kind --
impedances and resistances in circuits, crossing types in knot diagrams,
probabilities of failure and successful operation in reliability, bond
strengths in statistical mechanics. A natural way to think of these attributes
is to allow each element to carry two parameters, $a$ and $b$ say, with $a$
contributing to the terms of the Tutte polynomial corresponding to subsets
that include the given element, and $b$ contributing to the terms of the Tutte
polynomial corresponding to subsets that do not. Zaslavsky \cite{Z} calls the
resulting polynomial%
\begin{equation}
\sum_{T\subseteq W(G)}\left(
{\displaystyle\prod\limits_{t\in T}}
a(t)\right)  \left(
{\displaystyle\prod\limits_{w\notin T}}
b(w)\right)  s^{r^{G}(W(G))-r^{G}(T)}z^{\left\vert T\right\vert -r^{G}(T)}
\label{partutte}%
\end{equation}
the \emph{parametrized rank polynomial} of $M[IAS(G)]$; we denote it
$\tau(M[IAS(G)])$.

We do not give a general account of the theory of parametrized Tutte
polynomials here; the interested reader is referred to the literature, for
instance \cite{BR, EMT, Sok, T, Z}. However it is worth taking a moment to
observe that parametrized polynomials are very flexible, and the same
information can be formulated in many ways.\ For instance if $s$ and the
parameter values $b(w)$ are all invertible then formula (\ref{partutte}) is
equivalent to
\[
s^{r^{G}(W(G))}\cdot\left(
{\displaystyle\prod\limits_{w\in W(G)}}
b(w)\right)  \cdot\sum_{T\subseteq W(G)}\left(
{\displaystyle\prod\limits_{t\in T}}
\left(  \frac{a(t)}{b(t)s}\right)  \right)  (sz)^{\left\vert T\right\vert
-r^{G}(T)}\text{,}%
\]
which expresses $\tau(M[IAS(G)])$ as the product of a prefactor and a sum that
is essentially a parametrized rank polynomial with only $a$ parameters and one
variable, $sz$. We prefer formula (\ref{partutte}), though, because we do not
want to assume invertibility of the $b$ parameters.

Suppose that the various parameter values $a(w)$ and $b(w)$ are independent
indeterminates, and let $P$ denote the ring of polynomials with integer
coefficients in the $2+6\left\vert V(G)\right\vert $ independent
indeterminates $\{s,z\}\cup\{a(w)$, $b(w)\mid w\in W(G)\}$. Let $J$ be the
ideal of $P$ generated by the set of $4\left\vert V(G)\right\vert $ products
$\{a(v_{\phi})a(v_{\chi})$, $a(v_{\phi})a(v_{\psi})$, $a(v_{\chi})a(v_{\psi}%
)$, $b(v_{\phi})b(v_{\chi})b(v_{\psi})\mid v\in V(G)\}$, and let
$\pi:P\rightarrow P/J$ be the natural map onto the quotient. Then the only
summands of (\ref{partutte}) that make nonzero contributions to $\pi
\tau(M[IAS(G)])$ correspond to transversals of the vertex triangulation of
$W(G)$, i.e., subsets $T\subseteq W(G)$ with the property that $\left\vert
T\cap\{v_{\phi},v_{\chi},v_{\psi}\}\right\vert =1$ $\forall v\in V(G)$. We
denote the collection of all such transversals $\mathcal{T}(W(G))$. Each
$T\in\mathcal{T}(W(G))$ has $\left\vert T\right\vert =\left\vert
V(G)\right\vert =r^{G}(W(G))$, so $s$ and $z$ have the same exponent in the
corresponding term of (\ref{partutte}):
\[
\pi\tau(M[IAS(G)])=\pi\left(  \sum_{T\in\mathcal{T}(W(G))}\left(
{\displaystyle\prod\limits_{t\in T}}
a(t)\right)  \left(
{\displaystyle\prod\limits_{w\notin T}}
b(w)\right)  (sz)^{\left\vert V(G)\right\vert -r^{G}(T)}\right)  \text{.}%
\]

Observe that $\pi$ is injective when restricted to the additive subgroup $A$
of $P$ generated by products%
\[
\left(
{\displaystyle\prod\limits_{t\in T}}
a(t)\right)  \left(
{\displaystyle\prod\limits_{w\notin T}}
b(w)\right)  (sz)^{k}%
\]
where $k\geq0$ and $T\in\mathcal{T}(W(G))$. Consequently there is a
well-defined isomorphism of abelian groups $\pi^{-1}:\pi(A)\rightarrow A$, and
we have
\begin{equation}
\pi^{-1}\pi\tau(M[IAS(G)])=\sum_{T\in\mathcal{T}(W(G))}\left(
{\displaystyle\prod\limits_{t\in T}}
a(t)\right)  \left(
{\displaystyle\prod\limits_{w\notin T}}
b(w)\right)  (sz)^{\left\vert V(G)\right\vert -r^{G}(T)}\text{.}
\label{partutte2}%
\end{equation}

Note that $\pi^{-1}\pi\tau(M[IAS(G)])$, the image of the parametrized Tutte
polynomial $\tau(M[IAS(G)])$ under the mappings $\pi$\ and $\pi^{-1}$, might
also be described as the \emph{section} of $\tau(M[IAS(G)])$ corresponding to
$\mathcal{T}(W(G))$. Either way, formula (\ref{partutte2}) describes an
element of $P$, where $s$, $z$ and the various parameter values $a(w)$, $b(w)$
are all independent indeterminates.

Arratia, Bollob\'{a}s and Sorkin \cite{A} define the two-variable
\emph{interlace polynomial }$q(G)$ by the formula%
\begin{align*}
q(G)  &  =\sum_{S\subseteq V(G)}\left(  x-1\right)  ^{r(A(G)[S])}\left(
y-1\right)  ^{\left\vert S\right\vert -r(A(G)[S])}\\
&  =\sum_{S\subseteq V(G)}\left(  \frac{y-1}{x-1}\right)  ^{\left\vert
S\right\vert -r(A(G)[S])}\left(  x-1\right)  ^{\left\vert S\right\vert
}\text{.}%
\end{align*}
Here $r(A(G)[S])$ denotes the $GF(2)$-rank of the principal submatrix of
$A(G)$ involving rows and columns corresponding to vertices from $S$.

Let $\mathcal{T}_{0}(W(G))=\{T\in\mathcal{T}(W(G))\mid v_{\psi}\notin T$
$\forall v\in V(G)\}$, and for $T\in\mathcal{T}(W(G))$ let $S(T)=\{v\in
V(G)\mid v_{\chi}\in T\}$. Then $T\mapsto S(T)$ defines a bijection from
$\mathcal{T}_{0}(W(G))$ onto the power set of $V(G)$. As $r^{G}(T)$ is the
$GF(2)$-rank of the matrix%
\[
\left(  \text{columns }v_{\phi}\text{ with }v\notin S(T)\mid\text{columns
}v_{\chi}\text{ with }v\in S(T)\right)
\]
and the columns $v_{\phi}$ are columns of the identity matrix,
\[
r^{G}(T)=\left\vert V(G)\right\vert -\left\vert S(T)\right\vert
+r(A(G)[S(T)]).
\]
It follows that $q(G)$ may be obtained from $\pi^{-1}\pi\tau(M[IAS(G)])$ by
setting $a(v_{\phi})\equiv1$, $a(v_{\chi})\equiv x-1$, $a(v_{\psi})\equiv0$,
$b(v_{\phi})\equiv1$, $b(v_{\chi})\equiv1$, $b(v_{\psi})\equiv1$, $s=y-1$ and
$z=1/(x-1)$. These assignments are not unique; for instance the values of $s$
and $z$ may be replaced by $s=(y-1)/u$ and $z=u/(x-1)$ for any invertible $u$.

The reader familiar with the Tutte-Martin polynomials of isotropic systems
studied by\ Bouchet \cite{Bi3, B5} and the interlace polynomials introduced by
Aigner and van der Holst \cite{AH}, Courcelle \cite{C}, and the author
\cite{Tw, T6} will have no trouble showing that appropriate values for $s$,
$z$ and the $a$ and $b$ parameters yield all of these polynomials from the
parametrized rank polynomial $\tau(M[IAS(G)])$.

\end{document}